\documentclass[a4paper]{amsart}

\usepackage{amsmath,amsfonts,amssymb,amsthm,mathrsfs}
\usepackage{color} 
\usepackage[all]{xy} 
\usepackage[backref=page]{hyperref} 
\usepackage{enumitem} 

\numberwithin{equation}{section} 

\theoremstyle{plain}
\newtheorem{thm}{Theorem}[section]
\newtheorem{lem}[thm]{Lemma}
\newtheorem{prop}[thm]{Proposition}
\newtheorem{cor}[thm]{Corollary}
\newtheorem*{thm*}{Theorem}

\theoremstyle{definition}
\newtheorem{defi}[thm]{Definition}
\newtheorem{ex}[thm]{Example}

\theoremstyle{remark}
\newtheorem{remi}[thm]{Remark}

\theoremstyle{plain}

\newcommand{\HGH}{\mathcal{H}_{R}(G,H)} 
\newcommand{\HGHA}{\mathcal{H}_{R}(G,H,A,\alpha)} 
\newcommand{\HGHSym}{\mathcal{H}_{R}(S_3,S_2)} 
\newcommand{\HGHASym}{\mathcal{H}_{R}(S_3,S_2,A,\alpha)} 
\newcommand{\eH}{\mathrm{e}_H} 
\newcommand{\eeH}{\mathbf{e}_H} 

\newcommand{\IndHG}{\operatorname{Ind}_{H}^{G}}
\newcommand{\CoindHG}{\operatorname{Coind}_{H}^{G}}
\newcommand{\ResHG}{\operatorname{Res}_{H}^{G}}
\DeclareMathOperator{\Ind}{Ind}
\DeclareMathOperator{\Coind}{Coind}
\DeclareMathOperator{\Res}{Res}
\DeclareMathOperator{\Maps}{Maps}
\DeclareMathOperator{\Hom}{Hom}
\DeclareMathOperator{\End}{End}
\DeclareMathOperator{\Aut}{Aut}
\DeclareMathOperator{\Gr}{Gr}

\begin{document}

\title{Skew Hecke algebras}
\author{James Waldron}
\address{School of Mathematics and Statistics, Herschel Building, Newcastle University, Newcastle-upon-Tyne, NE1 7RU} 
\email{james.waldron@ncl.ac.uk}
\author{Leon Deryck Loveridge}
\address{Department of Science and Industry Systems, University of South-Eastern Norway, 3616 Kongsberg, Norway}
\address{Okinawa Institute of Science and Technology Graduate University,
1919-1 Tancha, Onna-son, Okinawa 904-0495, Japan}
\email{Leon.D.Loveridge@usn.no}
\date{\today}
\maketitle

\begin{abstract}
Let $G$ be a finite group, $H \le G$ a subgroup, $R$ a commutative ring, $A$ an $R$-algebra, and $\alpha$ an action of $G$ on $A$ by $R$-algebra automorphisms.
We study the associated \emph{skew Hecke algebra} $\mathcal{H}_{R}(G,H,A,\alpha)$, which is the convolution algebra of $H$-invariant functions from $G/H$ to $A$.

We prove for skew Hecke algebras a number of common generalisations of results about skew group algebras and results about Hecke algebras of finite groups.
We show that skew Hecke algebras admit a certain double coset decomposition.
We construct an isomorphism from $\mathcal{H}_{R}(G,H,A,\alpha)$ to the algebra of $G$-invariants in the tensor product $A \otimes \End_{R} ( \Ind_{H}^{G} R )$.
We show that if $|H|$ is a unit in $A$, then $\mathcal{H}_{R}(G,H,A,\alpha)$ is isomorphic to a corner ring inside the skew group algebra $A \rtimes G$.

Alongside our main results, we show that the construction of skew Hecke algebras is compatible with certain group-theoretic operations, restriction and extension of scalars, certain cocycle perturbations of the action, gradings and filtrations, and the formation of opposite algebras.
The main results are illustrated in the case where $G = S_3$, $H = S_2$, and $\alpha$ is the natural permutation action of $S_3$ on the polynomial algebra $R[x_1,x_2,x_3]$.
\end{abstract}


\section{Introduction}

\subsection{Skew Hecke algebras}
\label{sec: intro_skew_hecke}
Let $R$ be a commutative ring, $A$ an $R$-algebra, $G$ a finite group, $H \le G$ a subgroup, and $\alpha : G \to \Aut_{R\text{-alg}} (A)$ a group homomorphism from $G$ to the group of $R$-algebra automorphisms of $A$.
The \emph{skew Hecke algebra} $\HGHA$ is by definition the $R$-module
\[
	\Maps(G/H,A)^{H} := \{ \phi : G/H \to A \; | \; \phi(hgH) = \alpha_{h} \phi(gH) \text{ for all } h \in H , \, gH \in G/H \}
\]
equipped with the multiplication
\begin{equation}
\label{eqn: intro_convolution_product}
( \phi * \psi ) (gH) = \sum_{kH \in G/H} \phi(kH) \alpha_{k} \psi(k^{-1}gH) ,
\end{equation}
where the sum is over a set of representatives of the left cosets of $H$ in $G$.
These algebras were introduced by Baker in \cite{Baker1998}, where they are called \emph{twisted Hecke algebras}.
The definition of $\HGHA$ can be extended to the case of infinite $G$ with the assumption that $H$ is a `Hecke subgroup', meaning that each $H$-double coset is the union of finitely many left $H$-cosets.
As special cases, note that if $H = 1$ is the trivial subgroup then
\[
\mathcal{H}_{R} ( G , 1 , A , \alpha ) = A \rtimes G
\]
is the \emph{skew group algebra} \cite{PassmanBook1989} associated to the action $\alpha$ of $G$ on $A$; if $A = R$ and $\alpha = \mathrm{triv}$ is the trivial homomorphism then
\[
\mathcal{H}_{R}( G , H , R , \mathrm{triv} ) = \HGH 
\]
is the \emph{Hecke algebra} \cite{KriegBook1990} with coefficient ring $R$, for which $\HGH \cong R[H \backslash G / H]$ as $R$-modules.

\subsection{Motivation}
\subsubsection{Algebraic topology}
Skew hecke algebras first appeared in work of Baker on algebraic topology.
It is shown in \cite{Baker1998} that
\[
Ell_* \otimes \mathcal{H}_{\mathbb{Q}} ( \mathrm{GL}_2(\mathbb{Q}) , \mathrm{SL}_2(\mathbb{Z}) )
\cong
\mathcal{H}_{\mathbb{Q}} ( \mathrm{GL}_2(\mathbb{Q}) , \mathrm{SL}_2(\mathbb{Z}) , \widetilde{Ell}_* \otimes \mathbb{Q} , \alpha ) ,
\]
where $Ell_*$ and $\widetilde{Ell}_*$ rings of modular forms appearing in the theory of elliptic cohomology and $\alpha$ is an action of $\mathrm{GL}_{2}(\mathbb{Q})$; in \cite{Baker2024} skew Hecke algebras are related to the structure of certain modules related to Morava K-theory.

\subsubsection{C*-algebras}
In the setting of C*-algebras and Fell bundles, Palma has developed a theory of \emph{Hecke crossed products}, see \cite{Palma2012,Palma2013,PalmaBook2018}, extending the theory of Hecke C*-algebras.
The results in \emph{loc.\ cit}.\ are mostly distinct from those in our work, though there are some analogous statements which we will highlight when they arise, see e.g.\ Proposition \ref{mprop: stone} and Remark \ref{rem:main_results_generalisations} below.

\subsubsection{Quantum reference frames}
Our motivation comes from certain constructions in the theory of `quantum reference frames' \cite{BuschLM2018,FewsterJLRW2024,HametteGHLM2021}, where von Neumann algebras of observables of the form
\begin{equation}
\label{eqn:qrf_invariants}
( M \otimes B(L^{2}(G/H)) )^{G}
\end{equation}
arise, for $M$ a von Neumann algebra equipped with an action of a topological group $G$.
One aspect of the theory is the construction of certain `relativisation maps'
\begin{equation}
\label{eqn:qrf_yen}
M^{H} \to ( M \otimes B(L^{2}(G/H)) )^{G}
\end{equation}
as given concretely in \cite{GlowackiLW2024}.
For finite groups acting on general $R$-algebras, the algebra \eqref{eqn:qrf_invariants} is related to skew Hecke algebras by our Theorem \ref{mthm: fixed_points}, and the existence of an algebra homomorphism \eqref{eqn:qrf_yen} is given by Corollary \ref{cor: fixed_points_Yen}.

\subsubsection{Algebraic geometry}
Suppose that $X = \mathrm{Spec} \, S$ is an affine scheme equipped with an action $\alpha$ of a finite group $G$, and $N \le G$ is a normal subgroup.
Then $G/N$ acts naturally on the scheme $X/N = \mathrm{Spec} \, S^{N}$, and there is an associated quotient stack $(X/N) // (G/N)$ which can be understood at the level of quasi-coherent sheaves via the equivalence of categories
\[
\mathrm{QCoh}( (X/N) // (G/N) ) \cong S^{N} \rtimes G/N \mathrm{-Mod} .
\]
If $H \le G$ is not normal then the quotient stack cannot be defined, but we do still have the category
\[
\mathcal{H}_{\mathbb{Z}}(G,H,S,\alpha) \mathrm{-Mod}
\]
of modules over a skew Hecke algebra, and $\mathcal{H}_{\mathbb{Z}}(G,N,S,\alpha) \cong S^{N} \rtimes G/N$ so that this is consistent with the equivalence above.


\subsection{Main results}
\label{sec: intro_main_results}

\begin{thm}[{Thm.\ \ref{thm: double_coset}}]
\label{mthm: double_coset}
There is an isomorphism of $A^{H}$-bimodules
\begin{align*}
\HGHA & \xrightarrow{\cong} \bigoplus_{HgH} A^{H \cap gHg^{-1}} \\
\phi & \mapsto \left( \phi(g_1H) , \dots , \phi(g_nH) \right)
\end{align*}
where the sum is over a set of representatives of the double cosets of $H$ in $G$.
See \S \ref{sec: R_module_structure} for the definitions of the bimodule structures appearing in the statement.
\end{thm}

\begin{prop}[{Prop.\ \ref{prop:skew_hecke_corner}, Cor.\ \ref{cor:corner_ring}}]
\label{mthm: skew_hecke_corner}
If $|H|$ is a unit in $A$ then
\[
\eeH := \frac{1}{|H|} \sum_{h \in H} 1_{A} \otimes h
\]
is an idempotent in the skew group algebra $A \rtimes G$, and there is an $R$-algebra isomorphism
\[
\HGHA \cong \eeH ( A \rtimes G ) \eeH .
\]
\end{prop}

The first isomorphism in the following statement appears in recent work of Baker \cite[App.\ A]{Baker2024}.
Our proof is similar, except that we allow $A$ to be non-commutative and so must distinguish between left and right actions of $A$, and between $A$ and $A^{\mathrm{op}}$.

\begin{thm}[{Thm.\ \ref{thm: fixed_points}}]
\label{mthm: fixed_points}
There are $R$-algebra isomorphisms
\[
\HGHA \cong \End_{A \rtimes G}(A[G/H])^{\mathrm{op}} \cong \left( A \otimes \End_{R}(R[G/H])^{\mathrm{op}} \right)^{G} ,
\]
where $A \rtimes G$ acts on the free $A$-module $A[G/H]$ by $(ag) \cdot (bkH) = (a \alpha_{g}b)gkH$, and $G$ acts on the free $R$-module $R[G/H] \cong \IndHG R$ by $g \cdot kH = gkH$ and by conjugation on $\End_{R}(R[G/H])$.
An explicit isomorphism from $\HGHA$ to $(A \otimes \End_{R}(R[G/H]))^{G}$ is given by the map
\[
\phi \mapsto \sum_{gH,kH} \alpha_{k} (\phi(k^{-1}gH)) \otimes E_{gH,kH}
\]
where $E_{gH,kH}(kH) = gH$ and $E_{gH,kH}(lH) = 0$ for $lH \ne kH$.
\end{thm}

\begin{prop}[{Prop.\ \ref{prop: stone}}]
\label{mprop: stone}
Let $RG$ be the commutative algebra of $R$-valued functions on $G$, and $\alpha$ the action of $G$ on $RG$ given by $(\alpha_{g} f) (k) = f(g^{-1}k)$ for $g,k \in G$ and $f \in RG$.
There is an isomorphism of $R$-algebras
\[
\mathcal{H}_{R}(G,H,RG,\alpha) \cong \End_{R} (R[G/H]) .
\]
\end{prop}

\begin{remi}
\label{rem:main_results_generalisations}
Theorems \ref{mthm: double_coset}, \ref{mthm: skew_hecke_corner} and \ref{mthm: fixed_points} generalise the well-known isomorphisms for Hecke algebras $\HGH \cong R[H \backslash G / H]$, $\HGH \cong \eH R[G] \eH$ \cite[Chap.\ I, Thm.\ 6.6]{KriegBook1990}, and $\HGH \cong \End_{R[G]} (R[G/H])^{\mathrm{op}}$ \cite[Chap.\ I, Thm.\ 4.8]{KriegBook1990}.
Proposition \ref{mprop: stone} is a generalisation of the well-known isomorphism $RG \rtimes G \cong M_{n}(R)$. For $R = \mathbb{C}$ and $A$ a C*-algebra, Proposition \ref{mprop: stone} is a special case of the `Stone-von Neumann theorem' proven by Palma for Hecke crossed product C*-algebras \cite[Thm.\ 7.0.3]{PalmaBook2018}, see Remark \ref{remi: palma_stone} below.
\end{remi}

\subsection{Further results}
In \S \ref{sec: change_groups} we prove several results concerning the behaviour of skew Hecke algebras under certain group-theoretic operations:

\begin{enumerate}
\item If $N \le H \le G$ with $N$ normal in $G$ then there is an isomorphism
\[
\mathcal{H}_{R} (G,H,A,\alpha) \cong \mathcal{H}_{R} (G/N,H/N,A^{N},\alpha^{N}) .
\]

\item If $A_{1}$ and $A_{2}$ are free as $R$-modules, then a tensor product
\[
\mathcal{H}(G_1,H_1,A_1,\alpha_1) \otimes \mathcal{H}(G_2,H_2,A_2,\alpha_2) .
\]
is isomorphic to the skew Hecke algebra
\[
\mathcal{H}_{R}(G_1 \times G_2,H_1 \times H_2,A_1 \otimes A_2,\alpha_1 \otimes \alpha_2) .
\]

\item If $H \le K \le G$ then there is an injective $R$-algebra morphism
\[
\mathcal{H}_{R}(K,H,A,\alpha) \hookrightarrow \HGHA .
\]

\item If $H,H' \le G$ are conjugate subgroups of $G$ then there is an isomorphism
\[
\HGHA \cong \mathcal{H}_{R}(G,H',A,\alpha) .
\]

\item If $N \rtimes K$ is a semi-direct product group with $K$ finite, $H \le K$ is a subgroup, and $\alpha$ is the $R$-linear extension to $R[N]$ of the action of $K$ on $N$, then there is an $R$-algebra isomorphism
\[
\mathcal{H}_{R}(K,H,R[N],\alpha) \cong \mathcal{H}_{R}(N \rtimes K,H) .
\]
\end{enumerate}

\begin{remi}
Statements (a) and (b) in the preceding list generalise the well-known isomorphisms for Hecke algebras $\mathcal{H}_{R}(G,H) \cong \mathcal{H}_{R}(G/N,H/N)$ \cite[Chap.\ I, Cor.\ 6.2]{KriegBook1990} and $\mathcal{H}_{R}(G_1 \times G_2 , H_1 \times H_2) \cong \mathcal{H}_{R}(G_1,H_2) \otimes \mathcal{H}_{R}(G_2,H_2)$ \cite[Chap.\ I, Cor.\ 6.2, Thm.\ 6.3]{KriegBook1990}.
Statement (a) also generalises the fact that if $H \le G$ is normal then
\[
\HGHA \cong A^{H} \rtimes G/H ;
\]
see Proposition \ref{prop: special_cases}.
\end{remi}

In \S \ref{sec: further_properties} we show the following:

\begin{enumerate}
\item If $S$ is a commutative $R$-algebra free as an $R$-module, and $\alpha_{S}(g) = \mathrm{id}_{S} \otimes \alpha_{g}$, then
\[
S \otimes_{R} \HGHA \cong \mathcal{H}_{S}(G,H,S \otimes_{R} A,\alpha_S) .
\]

\item If $\beta_{g} = \chi(g)\alpha_{g}(-)\chi(g)^{-1}$ for $\chi : G \to A^{\times}$ an $\alpha$-cocycle satisfying $\chi(h) = 1_{A}$ for all $h \in H$, then
\[
\HGHA \cong \mathcal{H}_{R}(G,H,A,\beta) .
\]
In particular, if $\alpha_{g} a = \chi(g) a \chi(g)^{-1}$ for some group homomorphism $\chi : G \to A^{\times}$ satisfying $\chi(H) = \{1_{A}\}$, and $A$ is free as an $R$-module, then
\[
\HGHA \cong A \otimes \HGH .
\]

\item The construction of $\HGHA$ is compatible with gradings or filtrations on $A$, and if $A$ is filtered and $|G|$ is a unit in $A$ then
\[
\Gr \HGHA \cong \mathcal{H}_{R}(G,H,\Gr A,\alpha^{\mathrm{Gr}}) ,
\]
where `$\mathrm{Gr}$' denotes `associated graded algebra', see \S \ref{sec: grad_filt} for details.

\item If $\alpha^{\mathrm{op}}$ is the action of $G$ on the opposite algebra $A^{\mathrm{op}}$, defined by $\alpha_{g}^{\mathrm{op}} = \alpha_{g}$, then
\[
\HGHA^{\mathrm{op}} \cong \mathcal{H}_{R}(G,H,A^{\mathrm{op}},\alpha^{\mathrm{op}}) .
\]
\end{enumerate}

\begin{remi}
Statement (b) in the preceding list generalises the corresponding fact for skew group rings, and relates to the automorphisms of $\HGH$ determined by group homomorphisms $\chi : G \to R^{\times}$ with $H \le \ker \chi$ \cite[Chap.\ I, Lemm.\ 6.5]{KriegBook1990}.
Statement (d) is an analogue of the existence of the natural involution on $\HGH$ \cite[Chap.\ I, Ex.\ 7.3]{KriegBook1990}.
\end{remi}

\subsection{Examples}
\label{sec: main_results_examples}
In \S \ref{sec: example_S3} we apply our main results to skew Hecke algebras of the form
\[
\mathcal{H}_{R}(S_{3},S_{2},A,\alpha) ,
\]
$S_{2} \le S_{3}$ being the smallest example of a non-normal subgroup of a group.
We describe for general $A$ and $\alpha$ the $R$-module structure and product on $\HGHASym$ and then specialise to the case where $A = R[x_1,x_2,x_3]$ with $S_3$ acting by permuting the generators $x_1,x_2,x_3$.
We also consider the case where $A = RS_{3}$ is the commutative algebra of $R$-valued functions on $G$, and the case where $A = R[K]$ is the group algebra of a finite group.


\subsection{Terminology and notation}
\label{sec: intro_existing_results_terminology}
In order to stay close to established nomenclature in the ring theory literature \cite{MontgomeryBook1980,PassmanBook1989} we use the term `skew Hecke algebra' rather than `twisted Hecke algebra' or `Hecke crossed product'.
We view skew Hecke algebras as generalisations of skew group algebras and of Hecke algebras.
We use the notation $\HGHA$ as it is similar to the standard notation $\HGH$ for Hecke algebras.
\\
\\
Note that in the ring theory literature, a `twisted group ring' $R^{\omega}[G]$ has a product twisted by a 2-cocycle $\omega : G \times G \to R^{\times}$, and a `crossed product' is a similarly-twisted version of a skew group algebra; see \cite{PassmanBook1989} for further details.
In the operator algebra literature, the term `crossed product' (resp.\ `twisted crossed product') is typically used to refer to analogues of what in the ring theory literature is called a `skew group algebra' (resp.\ `crossed product'), e.g.\ compare \cite{PassmanBook1989} with \cite{WilliamsBook2007}.
\\
\\
The skew Hecke algebras defined above are distinct from the `twisted graded Hecke algebras' introduced by Witherspooon in \cite{Witherspoon2007}, the latter being more closely related to the `degenerate affine Hecke algebras' introduced by Drinfeld \cite{Drinfeld1986} and to deformations of the group algebras of Coxeter groups.

\subsection{Acknowledgements}
\label{sec: acknowledgements}
The authors would like to thank Andrew Baker for conversations about skew Hecke algebras.
LL would like to thank the Theoretical Visiting Sciences Programme at the Okinawa Institute of Science and Technology (OIST) for enabling his visit, and for the generous hospitality and excellent working conditions during his time there, which significantly aided the development and completion of this work.


\section{Background and notation}
\label{sec: setup_notation}

\subsection{Rings}
All rings have two-sided identities and unless otherwise stated all ring homomorphisms are identity preserving.
If $S$ is a ring and $X$ is a set then we use $S[X]$ to denote the free $S$-module on $X$.
\\
\\
Throughout, $R$ denotes a fixed commutative ground ring with unit element $1_R$.
We write $R^{\times}$ for the group of units in $R$.
Unless otherwise stated, $\Hom$, $\End$, $\otimes$ and $(-)^{*}$ denote $\Hom_{R}$, $\End_R$, $\otimes_R$ and $\Hom_{R}(-,R)$ respectively.
In \S \ref{sec: ext_scalars} we allow the ground ring to vary.
\\
\\
By an $R$-algebra, or just algebra when $R$ is understood, we mean a ring $A$ equipped with a ring homomorphism $R \to Z(A)$ from $R$ to the center $Z(A)$ of $A$;
equivalently, $A$ is a ring equipped with an $R$-module structure with respect to which the multiplication $A \times A \to A$ is $R$-bilinear.


\subsection{Groups}
\label{sec: setup_RG_modules}
Unless otherwise stated, $G$ will always denote a \emph{finite} group, and $H$ will denote a subgroup of $G$.
We denote by $R[G]$ the group algebra of $G$.
For $R$ the trivial $R[H]$-module, we identify the induced module $\IndHG R$ with the free $R$-module $R[G/H]$ with $G$-action $g \cdot kH = gkH$.
We use e.g.\ $\sum_{g}$ as a shorthand for $\sum_{g \in G}$, $\sum_{h}$ for $\sum_{h \in H}$, and $\sum_{gH}$ for a sum over a set of representatives of the left cosets of $H$ in $G$.

\subsection{Skew group algebras}
\label{sec: setup_skew_group}
For details about skew group algebras see \cite{PassmanBook1989}.
Let $G$ be a group, $A$ an $R$-algebra, and $\alpha: G \to \Aut_{R\text{-alg}}(A)$ a group homomorphism from $G$ to the group of $R$-algebra automorphisms of $A$.
The associated \emph{skew group algebra} $A \rtimes G$ or $A \rtimes_{\alpha} G$ is the free $A$-module $A[G]$ equipped with the $R$-bilinear multiplication determined by
\begin{equation}
\label{eqn: skew_group_mult}
(ag) (bk) = a (\alpha_g b) (gh) .
\end{equation}
Equivalently, $A \rtimes G = A \otimes R[G]$ with the product $(a \otimes g) (b \otimes k) = a (\alpha_g b) \otimes gh$.
There is an $R$-algebra isomorphism from $A \rtimes G$ to the $R$-module $\Maps(G,A)$ of functions from $G$ to $A$ equipped with the convolution product
\begin{equation}
\label{eqn: skew_group_convolution}
(\phi * \psi) (g) = \sum_{k \in G} \phi(k) \alpha_{k} \psi(k^{-1}g) .
\end{equation}
The isomorphism maps $ag$ to the function $\delta_{g,a}$ with value $a$ at $g$, and zero elsewhere.


\subsection{Hecke algebras}
\label{sec: setup_hecke}
For details about Hecke algebras see \cite{KriegBook1990}, and also \cite[\S 11 D]{CurtisRBook1990}, \cite{BumpBook1997}, \cite[Chap.\ 45-46]{BumpBook2013}.
Let $G$ be a finite group and $H \le G$ a subgroup.
The associated \emph{Hecke algebra} $\HGH$ is the $R$-module
\[
\Maps(G/H,R)^{H} := \{ \phi : G/H \to R \mid \phi(hgH) = \phi(gH) \text{ for all } h \in H , g \in G\}
\]
of $H$-invariant functions from $G/H$ to $R$, equipped with the product
\begin{equation}
\label{eqn: setup_hecke_product}
( \phi * \psi ) (gH) = \sum_{kH \in G/H} \phi(kH) \psi(k^{-1}gH)
\end{equation}
where the sum is over a set of representatives of the left cosets of $H$ in $G$.
There is an isomorphism of $R$-modules
\[
\HGH \cong R[H \backslash G / H]
\]
where $R[H \backslash G / H]$ is the free $R$-module on the set of $H$-double cosets in $G$. 
 There is an isomorphism of $R$-algebras 
\[
\HGH \cong \End_{R[G]} ( \IndHG R ) ^{\mathrm{op}}
\]
where $\IndHG R \cong R[G/H]$ is the induction of the trivial $R[H]$-module, see \cite[Chap.\ 1, Thm.\ 4.8]{KriegBook1990}.
If $|H|$ is a unit in $R$ then $\eH := \frac{1}{|H|} \sum_{h \in H} h$ is an idempotent in $R[G]$, and there is an $R$-algebra isomorphism
\[ 
\HGH \cong \eH R[G] \eH ,
\]
see \cite[Chap.\ 1, Thm.\ 6.6 \& Cor.\ 6.7]{KriegBook1990}.


\section{Skew Hecke algebras}
\label{sec: skew_hecke}

\subsection{Setup}
\label{sec: setup_tuples}
Throughout the remainder of the paper a tuple $(G,H,A,\alpha)$ will denote the data of a \emph{finite} group $G$ with identity element $1_G$ or just $1$ when $G$ is understood, a subgroup $H \le G$, an $R$-algebra $A$, and a group homomorphism $\alpha : G \to \Aut_{R\text{-alg}}(A)$ from $G$ to the group of $R$-algebra automorphisms of $A$.
We write $\alpha_g$ for $\alpha(g)$ and $\alpha_g a$ for $\alpha_g (a)$.
Associated to this data is the skew group algebra $A \rtimes G$ defined in \S \ref{sec: setup_skew_group} and the Hecke algebra $\HGH$ defined in \S \ref{sec: setup_hecke}.

\subsection{The $R$-module structure and product}
\label{sec: the_product}
The following definition appears in work of Baker \cite{Baker1998} under the name `twisted Hecke algebra', and work of Palma \cite{PalmaBook2018} under the name `Hecke crossed product'.
See Remark \ref{sec: intro_existing_results_terminology} for an explanation of our choice of terminology.

\begin{defi}[{\cite[\S 2]{Baker1998},\cite[\S 3]{PalmaBook2018}}]
\label{def: skew_hecke}
The \emph{skew Hecke algebra} $\HGHA$ associated to the data $(G,H,A,\alpha)$ is the $R$-module
\begin{equation}
\label{eqn: def_skew_hecke}
\Maps(G/H,A)^{H} := \{ \phi : G/H \to A \; | \; \phi(hgH) = \alpha_{h} \phi(gH) \; \forall h \in H \; , \; gH \in G/H \}
\end{equation}
equipped with the product
\begin{equation}
\label{eqn: skew_hecke_product}
( \phi * \psi ) (gH) = \sum_{kH} \phi(kH) \alpha_{k} \psi(k^{-1}gH)
\end{equation}
where the sum in \eqref{eqn: skew_hecke_product} is over a set of representatives of the left cosets of $H$ in $G$.
\end{defi}

As explained in \cite[\S 3]{PalmaBook2018} and \cite[\S 2]{Baker1998}, it follows from straightforward computations using the fact that $\phi , \psi$ are $H$-invariant, that the right hand side of \eqref{eqn: skew_hecke_product} is well defined, $\phi * \psi$ is $H$-invariant, and $*$ is associative.
The algebra $\HGHA$ is unital, with the identity element equal to the function $\delta_{H,1_{A}}$ with value $1_{A}$ at the identity coset $H$, and zero elsewhere.
Note that as a ring, $\HGHA$ is independent of the ground ring $R$ or the $R$-algebra structure on $A$.

\begin{prop}
\label{prop: skew_hecke_R_module_isomorphisms}
The following $R$-modules are isomorphic.
\begin{enumerate}
\item \label{it: skew_hecke_R_module_HGHA} $\HGHA$
\item \label{it: skew_hecke_R_module_Maps_GH_A} $\Maps(G/H,A)^{H}$
\item \label{it: skew_hecke_R_module_A_RGH} $( A \otimes R[G/H] )^{H}$
\item \label{it: skew_hecke_R_module_A_GH} $A[G/H]^{H}$
\item \label{it: skew_hecke_R_module_Ind} $(\Ind_H^G \Res_H^G A)^{H}$
\item \label{it: skew_hecke_R_module_Coind} $(\Coind_H^G \Res_H^G A)^{H}$
\item \label{it: skew_hecke_R_module_Maps_G_A} $\Maps(G,A)^{H \times H}$
\item \label{it: skew_hecke_R_module_A_RG} $( A \otimes R[G] )^{H \times H}$
\item \label{it: skew_hecke_R_module_A_G} $A[G]^{H \times H}$
\end{enumerate}
where in \eqref{it: skew_hecke_R_module_Maps_GH_A}, \eqref{it: skew_hecke_R_module_A_RGH}, \eqref{it: skew_hecke_R_module_A_GH}, \eqref{it: skew_hecke_R_module_Maps_G_A}, \eqref{it: skew_hecke_R_module_A_RG}, \eqref{it: skew_hecke_R_module_A_G} the actions of $H$ or $H \times H$ are respectively
\begin{align}
\label{eqn: action_Maps_GH_A}
( h \cdot \phi ) (gH) & = \alpha_{h} \phi(h^{-1}gH) & \phi \in \Maps(G/H,A) \\
\label{eqn: action_A_RGH}
h \cdot (a \otimes gH) & = \alpha_{h} a \otimes hgH & a \otimes gH \in A \otimes R[G/H] \\
\label{eqn: action_A_GH}
h \cdot (a (gH)) & = \alpha_{h} a \otimes hgH & a (gH) \in A[G/H] \\
\label{eqn: action_Maps_G_A}
( (h,h') \cdot \Phi ) (g) & = \alpha_{h} \Phi( h^{-1} g h' ) & \Phi \in \Maps(G,A) \\
\label{eqn: action_A_RG}
(h,h') \cdot ( a \otimes g ) & = \alpha_{h} a \otimes hgh'^{-1} & a \otimes g \in A \otimes R[G] \\
\label{eqn: action_A_G}
(h,h') \cdot (ag) & = (\alpha_{h} a) hg(h')^{-1} & ag \in A[G] .
\end{align}
\end{prop}

\begin{proof}
\eqref{it: skew_hecke_R_module_HGHA} = \eqref{it: skew_hecke_R_module_Maps_GH_A} by Definition \ref{def: skew_hecke};
\eqref{it: skew_hecke_R_module_Maps_GH_A} $\cong$ \eqref{it: skew_hecke_R_module_A_RGH} via the map $\phi \mapsto \sum_{gH} \phi(gH) \otimes gH$;
\eqref{it: skew_hecke_R_module_A_RGH} $\cong$ \eqref{it: skew_hecke_R_module_A_GH} is immediate;
 \eqref{it: skew_hecke_R_module_Ind} $\cong$ \eqref{it: skew_hecke_R_module_Coind} because $G$ is finite so that $\IndHG \cong \CoindHG$ \cite[\S 3.3]{BensonBook1998};
\eqref{it: skew_hecke_R_module_Coind} = \eqref{it: skew_hecke_R_module_Maps_G_A} by the definition of $\CoindHG$;
\eqref{it: skew_hecke_R_module_Maps_G_A} $\cong$ \eqref{it: skew_hecke_R_module_A_RG} via the map $\Phi \mapsto \sum_{g} \Phi(g) \otimes g$;
\eqref{it: skew_hecke_R_module_A_RG} $\cong$ \eqref{it: skew_hecke_R_module_A_G} is immediate.
Finally, \eqref{it: skew_hecke_R_module_Maps_GH_A} $\cong$ \eqref{it: skew_hecke_R_module_Maps_G_A} because functions $\phi : G/H \to A$ correspond to functions $\Phi : G \to A$ satisfying $\Phi(gh) = \Phi(g)$ for all $g \in G$, $h \in H$.
\end{proof}

Under the isomorphism $\phi \mapsto \sum_{gH} \phi(gH) gH$ from $\Maps(G/H,A)^{H}$ to $A[G/H]^{H}$, the convolution product \eqref{eqn: skew_hecke_product} corresponds to the product on $A[G/H]^{H}$ defined by
\begin{equation}
\label{eqn: skew_hecke_product_free}
\left( \sum_{gH} a_{gH} gH \right) \left( \sum_{kH} b_{kH} kH \right) = \sum_{gH,kH} a_{gH} ( \alpha_{gH} b_{kH} ) gkH ,
\end{equation}
and under the isomorphism $A[G/H]^{H} \cong (A \otimes R[G/H])^{H}$ this corresponds to
\begin{equation}
\label{eqn: skew_hecke_product_tensor}
\left( \sum_{gH} a_{gH} \otimes gH \right) \left( \sum_{kH} b_{kH} \otimes kH \right) = \sum_{gH,kH} a_{gH} ( \alpha_{gH} b_{kH} ) \otimes gkH .
\end{equation}

\subsection{The $|H|$ a unit case}
If $|H|$ is a unit in $A$ then the map
\begin{equation}
\label{eqn:isom_MapsGH,MapsG}
\phi \mapsto \left( g \mapsto \frac{1}{|H|} \phi(gH) \right)
\end{equation}
is an $R$-algebra isomorphism from $\Maps(G/H,A)^{H}$ to
\[
\Maps(G,A)^{H \times H} := \left\{ \Phi : G \to A \; | \; \Phi(hgh') = \alpha_{h} \Phi(g) \text{ for all } h,h' \in H, \; g \in G \right\} ,
\]
where $\Maps(G,A)^{H \times H}$ is equipped with restriction of the convolution product \eqref{eqn: skew_group_convolution} on $\Maps(G,A) \cong A \rtimes G$, so that
\[
(\Phi * \Psi)(g) = \sum_{g \in G} \Phi(k) \alpha_{k} \Psi(k^{-1}g) .
\]
Note that the factor of $1/|H|$ is needed to make the map \eqref{eqn:isom_MapsGH,MapsG} multiplicative and unital.
If $A = R$ then this isomorphism reduces to the corresponding result for Hecke algebras is \cite[Chap.\ I, Cor.\ 6.7]{KriegBook1990}.
The following result is a generalisation of the isomorphism $\HGH \cong \eH R[G] \eH$, see \cite[Chap.\ I, Thm.\ 6.6]{KriegBook1990}.

\begin{prop}
\label{prop:skew_hecke_corner}
If $|H|$ is a unit in $A$ then $\eeH := \frac{1}{|H|} \sum_{h \in H} 1_{A} \otimes h$ is an idempotent in the skew group ring $A \rtimes G$ and there is an $R$-algebra isomorphism
\begin{equation}
\label{eqn: skew_hecke_corner_isom}
\HGHA \cong \eeH ( A \rtimes G ) \eeH .
\end{equation}
\end{prop}

\begin{proof}
That $\eeH$ is an idempotent in $A \rtimes G$ follows from the fact that $(1_{A} h) (1_{A} h') = 1_{A} hh'$ for $h,h' \in H$.
Under the $R$-algebra isomorphism $\Maps(G,A) \cong A \rtimes G$ of \S \ref{sec: setup_skew_group}, the action \eqref{eqn: action_Maps_G_A} of $H \times H$ on $\Maps(G,A)$ corresponds to the action \eqref{eqn: action_A_RG} of $H \times H$ on $A \rtimes G$ given by
\begin{align*}
(h,h') \cdot (a \otimes g) & = \alpha_h a \otimes hgh'^{-1} = (1_A \otimes h) (a \otimes g) (1_A \otimes h'^{-1}) .
\end{align*}
It follows that the $H \times 1$-invariants in $A \rtimes G$ are $\eeH (A \rtimes G)$, the $1 \times H$ invariants are $(A \rtimes G) \eeH$, and as the actions of the two copies of $H$ commute, the $H \times H$ invariants are $(A \rtimes G)^{H \times H} = \eeH (A \rtimes G) \eeH$.
\end{proof}

\begin{remi}
We give an alternative proof of Proposition \ref{prop:skew_hecke_corner} in Corollary \ref{cor:corner_ring} using the isomorphism $\HGHA \cong \End_{A \rtimes G}(A[G/H])^\mathrm{op}$ of Theorem \ref{thm: fixed_points}.
\end{remi}


\subsection{Two subalgebras}
\label{sec: two_subalgebras}
There are natural $R$-algebra homomorphisms
\begin{align}
\label{eqn: inclusion_general_AH}
A^{H} \to \HGHA & \; , \; a \mapsto \delta_{H,a} \\
\label{eqn: inclusion_general_HGH}
\HGH \to \HGHA & \; , \; \phi \mapsto i_{A} \circ \phi
\end{align}
where $\delta_{H,a}$ is the function with value $a$ at $H \in G/H$ and zero elsewhere, and $i_{A} : R \to A$ is the $R$-module structure map.
Under the isomorphisms $\HGH \cong R[H \backslash G / H]^{H}$ and $\HGHA \cong A[G/H]^{H}$, these morphisms correspond to the maps
\[
a \mapsto aH \;,\; \sum_{gH} r_{gH} gH \mapsto \sum_{gH} i_{A}(r_{gH}) gH .
\]
The morphism \eqref{eqn: inclusion_general_AH} is injective, and corresponds to the inclusion of the summand $A^{H}$ in the decomposition \eqref{eqn: double_coset} of Theorem \ref{thm: double_coset} below, whereas \eqref{eqn: inclusion_general_HGH} is injective if $A$ is faithful as an $R$-module.
As an $R$-algebra, the image of \eqref{eqn: inclusion_general_HGH} is isomorphic to $\mathcal{H}_{i_{A} (R)} (G,H)$, where $i_{A} : R \to A$ is the $R$-module structure map of $A$.


\subsection{Some special cases}
\label{sec: special_cases}
The following proposition shows that the skew-Hecke algebra $\HGHA$ is a common generalisation of the skew-group algebra $A \rtimes G$, the Hecke algebra $\HGH$, and the algebra of $G$-invariants $A^{G}$.
Most of the special cases in Proposition \ref{prop: special_cases} appear in \cite{Baker1998}, or in a different form in the work of Palma on Hecke crossed product C*-algebras \cite[\S 3.2]{PalmaBook2018},\cite[\S 3.2]{Palma2012}.

\begin{prop}[{\cite{Baker1998,PalmaBook2018}}]
\label{prop: special_cases}
The following $R$-algebra isomorphisms exist, where $1$ denotes a trivial group or the trivial subgroup of a given group, $\mathrm{triv}$ is the trivial homomorphism from $G$ to $\Aut_{k \text{-alg}}(A)$.
\begin{enumerate}
\item \label{it: spec_case.R}
$\mathcal{H}_{R}(G,H,R,\mathrm{triv}) \cong \HGH$.

\item \label{it: spec_case.11}
$\mathcal{H}_{R}(1,1,A,\mathrm{triv}) \cong A$.

\item \label{it: spec_case.a_triv}
$\mathcal{H}_{R}(G,H,A,\mathrm{triv}) \cong A \otimes \HGH$ if $A$ is a free $R$-module.

\item \label{it: spec_case.H_eq_G}
$\mathcal{H}_{R}(G,G,A,\alpha) \cong A^{G}$.

\item \label{it: spec_case.H_eq_1}
$\mathcal{H}_{R}(G,1,A,\alpha) \cong A \rtimes G$.

\item \label{it: spec_case.N}
$\mathcal{H}_{R}(G,H,A,\alpha) \cong A^{H} \rtimes G/H$ if $H \le G$ is normal.
\end{enumerate}
In statement \eqref{it: spec_case.N} the action of $G$ on $A$ preserves $A^{H}$ because $H$ is normal, and the action $\alpha^{H}$ of $G/H$ on $A^{H}$ is defined by $gH \cdot a = \alpha_{g} a$.
\end{prop}

\begin{proof}
Statements \eqref{it: spec_case.R},\eqref{it: spec_case.11} and \eqref{it: spec_case.H_eq_1} follow immediately from the definitions of the algebras $\HGHA$, $\HGH$ and $A \rtimes G$.
Statement \eqref{it: spec_case.H_eq_G} is the $H = G$ case of statement \eqref{it: spec_case.N} proven below.
\\
\\
\eqref{it: spec_case.a_triv}. 
If $G$ acts trivially on $A$ and $A$ is a free $R$-module then
\[
( A \otimes R[G/H] )^{H} = A \otimes R[G/H]^{H}
\]
and the product written in terms of tensors factorises as
\[
\left( \sum_{lH} a_{lH} \otimes lH \right) \cdot \left( \sum_{kH} b_{kH} \otimes kH \right) = \sum_{lH,kH} a_{lH} b_{kH} \otimes lkH .
\]
\\
\eqref{it: spec_case.N}.
If $H \le G$ is normal then $H$ acts trivially on $R[G/H]$, and because $R[G/H]$ is a free $R$-module
\[
( A \otimes R[G/H] )^{H} = A^{H} \otimes R[G/H]
\]
with product
\begin{align*}
\left( \sum_{lH} a_{lH} \otimes lH \right) \cdot \left( \sum_{kH} b_{kH} \otimes kH \right) & = \sum_{lH,kH} a_{lH} ( \alpha_{l} b_{kH} ) \otimes lkH \\
& = \sum_{lH,kH} a_{lH} ( \alpha^{H}_{lH} b_{kH} ) \otimes lkH ,
\end{align*}
where $\alpha^{H}_{lH} = \alpha_{l}$ as in the statement of the Proposition.
\end{proof}

\begin{remi}
\label{rem: reduce_to_cases}
It follows from Proposition \ref{prop: special_cases} that in order to construct a skew Hecke algebra $\HGHA$ that does \emph{not} have an immediate description in terms of fixed point algebras, skew group algebras, or Hecke algebras, one needs that $G$ is non-abelian, $H \le G$ is not normal, and the action $\alpha$ is non-trivial.
In \S \ref{sec: example_S3} below we consider the skew Hecke algebra
\[
\mathcal{H}_{R} (S_3,S_{2},R[x_1,x_2,x_3],\alpha),
\] 
where $\alpha$ is the action of the symmetric group $S_3$ on the polynomial ring $R[x_1,x_2,x_3]$ given by permuting the variables $x_1,x_2,x_3$.
\end{remi}


\section{Double coset decomposition}
\label{sec: R_module_structure}
We prove a common generalisation of the facts that $A \rtimes G \cong A^{\oplus |G|}$ as left $A$-modules and $\HGH \cong R[H \backslash G / H]$ as $R$-modules.
\\
\\
Consider $\HGHA$ as an $A^{H}$-bimodule via the $R$-algebra morphism \eqref{eqn: inclusion_general_AH}, so that $a \cdot \phi \cdot a' = \delta_{H,a} * \phi * \delta_{H,a'}$ for $a,a' \in A^{H}$ and $\phi \in \HGHA$.
For each $g \in G$, the $R$-submodule $A^{H \cap gHg^{-1}} \subseteq A$ is an $A^{H}$-bimodule with actions
\begin{equation}
\label{eqn: AH_bimodule_actions}
a \cdot b \cdot a' = a b ( \alpha_{g} a')
\end{equation}
for $a,a' \in A^{H}$ and $b \in A^{H \cap gHg^{-1}}$.
Note that if $a,b$ are as in \eqref{eqn: AH_bimodule_actions} and $h \in H \cap gHg^{-1}$, then $h = gh'g^{-1}$ for some $h' \in H$ and
\[
h \cdot \left( a b (\alpha_{g}a') \right) = a b ( \alpha_{hg} a' ) = a b ( \alpha_{gh'} a' ) = a b ( \alpha_{g} a' )
\]
so that $ab(\alpha_{g}a')$ is indeed contained in $A^{H \cap gHg^{-1}}$.

\begin{thm}
\label{thm: double_coset}
There is an isomorphism of $A^{H}$-bimodules
\begin{align}
\label{eqn: double_coset}
\HGHA & \xrightarrow{\cong} \bigoplus_{HgH} A^{H \cap gHg^{-1}} \\
\nonumber
\phi & \mapsto \left( \phi(g_1H) , \dots , \phi(g_nH) \right)
\end{align}
where the sum is over a set of representatives of the double cosets of $H$ in $G$.
\end{thm}

\begin{proof}
Under the left action of $H$ on $G/H$, the orbits are the double cosets of $H$ in $G$, and the stabiliser of a left coset $gH$ is the subgroup $H \cap gHg^{-1}$.
If $\phi \in \HGHA$ then by Definition \ref{def: skew_hecke}
\[
\alpha_{h} \phi(gH) = \phi(hgH) = \phi(gH)
\]
for all $h \in H \cap gHg^{-1}$, so that $\phi(gH) \in A^{H \cap gHg^{-1}}$, and an $H$-invariant function $G/H \to A$ is determined by its restriction to a set of representatives of the $H$-orbits.
This shows that \eqref{eqn: double_coset} is a bijection, and it is clear that this map is an $R$-module homomorphism.
To show that the map \eqref{eqn: double_coset} is a morphism of $A^{H}$-bimodules we calculate
\begin{align*}
( \delta_{H,a} * \phi * \delta_{H,a'} ) (gH) & = \sum_{kH} \delta_{H,a}(kH) \alpha_{k} \left( \phi * \delta_{H,a'} (k^{-1}gH) \right) \\
& = a \left( \phi * \delta_{H,a'} (gH) \right) \\
& = a \sum_{lH} \phi(lH) \alpha_{l} \delta_{H,a'} (l^{-1}gH) \\
& = a \phi(gH) (\alpha_{g} a') .
\end{align*}
\end{proof}

\begin{remi}
\label{rem: double_coset}
At the level of $R$-modules, the isomorphism of Theorem \ref{thm: double_coset} is essentially Mackey's subgroup theorem \cite[Thm.\ 10.13]{CurtisRBook1990} applied to the $R[H]$-module $\ResHG A$.
\end{remi}

\begin{remi}
It follows from Theorem \ref{thm: double_coset} that the map
\[
\HGHA \to A^{H} \; , \; \phi \mapsto \phi(H)
\]
is a morphism of $A^{H}$-bimodules, this map corresponding to the projection onto the summand $A^{H}$ on the right hand side of \eqref{eqn: double_coset}.
If $A = R$ and $\alpha$ is trivial then this map is equal to the `trace map'
\[
\mathrm{tr} : \HGH \to R \; , \; \phi \mapsto \phi(H)
\]
defined in \cite[Chap.\ 1, \S 8]{KriegBook1990}, whereas if $H = 1$ this map is equal to the $A$-bimodule morphism
\[
A \rtimes G \to A \; , \; \sum_{g} a_{g} \otimes g \mapsto a_{1} ,
\]
which in the operator algebra literature is often referred to as the `canonical conditional expectation' on $A \rtimes G$, see e.g.\ \cite{Itoh1981}.
\end{remi}

\section{Skew Hecke algebras as fixed point algebras}
\label{sec: fixed_point_algebra}
Let $A[G/H] \cong A \otimes \IndHG R$ be the free left $A$-module on the set $G/H$.
This is an $A \rtimes G$-module with action
\[
(ag) \cdot (bkH) = a (\alpha_{g}b) (gkH) ,
\]
and by restriction $A[G/H]$ is an $A$-module and an $R[G]$-module.
The group $G$ acts $R[G/H] \cong \IndHG R$ by $g \cdot kH = gkH$, and on the $R$-algebra $\End_{R}(R[G/H]) = \End_{R}(\IndHG R)$ by conjugation.
\\
\\
The first isomorphism in the following statement also appears in recent work of Baker \cite[App.\ A]{Baker2024}.
The proof we give is similar, except that we allow $A$ to be non-commutative and so must distinguish between left and right actions of $A$, and between $A$ and $A^{\mathrm{op}}$.

\begin{thm}[{\cite[App.\ A]{Baker2024}}]
\label{thm: fixed_points}
There are $R$-algebra isomorphisms
\[
\HGHA \cong \End_{A \rtimes G}(A[G/H])^{\mathrm{op}} \cong \left( A \otimes \End_{R}(R[G/H])^{\mathrm{op}} \right)^{G} .
\]
\end{thm}

\begin{proof}
We have
\begin{align*}
\End_{A \rtimes G}(A[G/H]) & = \End_{A}(A[G/H])^{G} ,
\end{align*}
where $G$ acts by conjugation on $\End_{A}(A[G/H])$, and there are natural $R$-module isomorphisms
\begin{align*}
\End_{A}(A[G/H])^{G} & = \Hom_{A}(A[G/H],A[G/H])^{G} \\
& \cong \Hom_{R}(R[G/H],A[G/H])^{G} & \text{($A[G/H]$ is free over $A$)} \\
& \cong \Hom_{R}(R,A[G/H])^{H} & \text{($\IndHG$ is left adjoint to $\ResHG$)} \\
& \cong A[G/H]^{H} & \text{($R$ is a trivial $R[H]$-module)}
\end{align*}
the composition of which is the $R$-module isomorphism
\begin{equation}
\label{eqn:isom_End_Hecke}
\End_{A}(A[G/H])^{G} \to A[G/H]^{H} \;,\; \Phi \mapsto \Phi(1_{A}H) .
\end{equation}
We show that the map \eqref{eqn:isom_End_Hecke} is a ring homomorphism from $\End_{A}(A[G/H])^{G}$ to $A[G/H]^{H}$, where $A[G/H]^{H}$ is equipped with the product \eqref{eqn: skew_hecke_product_free}.
Denote by $\circ$ the multiplication in $\End_{A}(A[G/H])$, and $\circ^{\mathrm{op}}$ its opposite.
An endomorphism $\Phi \in \End_{A}(A[G/H])$ can be decomposed uniquely in the form
\[
\Phi = \sum_{gH,kH} R_{\Phi_{gH,kH}} E_{gH,kH}
\]
where $E_{gH,kH}$ is the permutation mapping $kH$ to $gH$ and all other cosets to zero, $\Phi_{gH,kH} \in A$, and for $a \in A$ the map $R_{a}$ is the \emph{right} action of $a$, i.e.\ $R_{a} (b(lH)) = (ba) (lH)$.
In terms of this decomposition, $\Phi$ is $G$-invariant iff
\[
\alpha_{s} \Phi_{s^{-1}gH,s^{-1}kH} = \Phi_{gH,kH}
\]
for all $s \in G$ and $gH,kH \in G/H$.
If two elements $\Phi,\Psi \in \End_{A}(A[G/H])^{G}$ are decomposed in this way then
\[
\Phi(1_{A}H) = \sum_{gH} \Phi_{gH,H} gH \; , \; \Psi(1_{A}H) = \sum_{gH} \Psi_{gH,H} sH 
\]
and
\begin{align*}
(\Phi \circ^{\mathrm{op}} \Psi)(1_{A}H) & = (\Psi \circ \Phi)(1_{A}H) \\
& = \left( \sum_{sH,tH} \sum_{gH,kH} R_{\Psi_{sH,tH}} R_{\Phi_{gH,kH}} E_{sH,tH} E_{gH,kH} \right) (1_{A}H) \\
& = \sum_{sH,gH} \Phi_{gH,H} \, \Psi_{sH,gH} sH \\
& = \sum_{sH,gH} \Phi_{gH,H} \, (\alpha_{g} \Psi_{g^{-1}sH,H}) sH \\
& = \sum_{sH,gH} \Phi_{gH,H} \, (\alpha_{g} \Psi_{sH,H}) (gsH) \\
& = ( \sum_{gH} \Phi_{gH,H} gH ) ( \sum_{sH} \Psi_{sH} sH ) \\
& = \Phi(1_{A}H) * \Psi(1_{A}H)
\end{align*}
where the fourth equality follows from the fact that $\Psi$ is $G$-invariant, and the last equality follows from the formula \eqref{eqn: skew_hecke_product_free} for the product on $A[G/H]^{H}$ corresponding to the convolution product on $\HGHA$.
\\
\\
We also have an $R$-algebra isomorphism
\[
\Hom_{A}(A[G/H],A[G/H])^{G} \cong \left( A^{\mathrm{op}} \otimes \End_{R}(R[G/H]) \right)^{G}
\]
because $R[G/H]$ is a finite rank free $R$-module, and therefore
\[
\left( \Hom_{A}(A[G/H],A[G/H])^{G} \right)^{\mathrm{op}} \cong \left( A \otimes \End_{R}(R[G/H])^{\mathrm{op}} \right)^{G} .
\]
Inverting the isomorphism \eqref{eqn:isom_End_Hecke} then gives an explicit $R$-algebra isomorphism
\begin{equation}
\label{eqn: fixed_points_isom_tensor}
\sum_{gH} a_{gH} gH \mapsto \sum_{gH,kH} \alpha_{k} a_{k^{-1}gH} \otimes E_{gH,kH}
\end{equation}
from $A[G/H]^{H}$ to $(A \otimes \End_{R}(R[G/H])^{\mathrm{op}})^{G}$.
\end{proof}

\begin{remi}
Theorem \ref{thm: fixed_points} reduces to the isomorphism
\[
\HGH \cong \End_{R[G]} (\IndHG R)^{\mathrm{op}} = ( \End_{R} ( \IndHG R)^{\mathrm{op}} )^{G}
\]
if $A = R$ (this is \cite[Chap.\ 1, Thm.\ 4.8]{KriegBook1990}), and to the isomorphism
\[
A \rtimes G \cong \left( A \otimes \End_{R} (R[G])^{\mathrm{op}} \right)^{G}
\]
if $H = 1$ (this is e.g.\ \cite[Prop.\ 2.1]{AriasL2010}).
\end{remi}

Theorem \ref{thm: fixed_points} also gives the following short proof of Proposition \ref{prop:skew_hecke_corner}.

\begin{cor}
\label{cor:corner_ring}
If $|H|$ is a unit in $A$ then $\eeH := \frac{1}{|H|} \sum_{h \in H} 1_{A} h$ is an idempotent in $A \rtimes G$, and there is an $R$-algebra isomorphism
\[
\HGHA \cong \eeH (A \rtimes G) \eeH .
\]
\end{cor}

\begin{proof}
That $\eeH$ is an idempotent in $A \rtimes G$ follows from the fact that $(1_{A} h) (1_{A} h') = 1_{A} hh'$ for $h,h' \in H$.
As a left $A \rtimes G$-module, $A[G/H] \cong (A \rtimes G) \eeH$.
Then
\[
\End_{A \rtimes G}(A[G/H])^{\mathrm{op}} \cong \End_{A \rtimes G}( (A \rtimes G) \eeH )^{\mathrm{op}} \cong \eeH (A \rtimes G) \eeH
\]
and the statement follows from Theorem \ref{thm: fixed_points}.
\end{proof}

\begin{remi}
As $\IndHG R$ is a free $R$-module of rank $n := |G/H|$, Theorem \ref{thm: fixed_points} implies that
\[
\HGHA \cong \left( A \otimes M_{n}(R) \right)^{G} \cong M_{n}(A)^{G}
\]
where $G$ acts on $M_{n}(R)$ via the isomorphism $M_{n}(R) \cong \End_{R}(\IndHG R)$ determined by choosing an ordering of the cosets of $H$.
In particular, $\HGHA$ is isomorphic to a $G$-fixed point subring of a ring Morita equivalent to $A$.
\end{remi}

\begin{cor}
\label{cor: fixed_points_Yen}
With notation as in the proof of Theorem \ref{thm: fixed_points}, there is an injective $R$-algebra morphism
\begin{align}
\label{eqn: fixed_points_Yen}
A^{H} & \to \left( A \otimes \End_{R} (R[G/H]) \right) ^{G} \\
\nonumber
a & \mapsto \sum_{gH} \alpha_{g} a \otimes E_{gH,gH} .
\end{align}
\end{cor}

\begin{proof}
The map \eqref{eqn: fixed_points_Yen} in the statement is equal to the composition of the $R$-algebra morphism $A^{H} \to \HGHA$, $a \mapsto a \otimes H$ (see \S \ref{sec: two_subalgebras}) and the isomorphism \eqref{eqn: fixed_points_isom_tensor} of Theorem \ref{thm: fixed_points}.
\end{proof}

\begin{remi}
\label{rem: Yen_QRF}
The morphism \eqref{eqn: fixed_points_Yen} of Corollary \ref{cor: fixed_points_Yen} appears as the relativisation map `$\yen$' in joint work of the authors and G\l{}owacki \cite[Prop.\ III.1]{GlowackiLW2024}.
\end{remi}


\section{Functions on $G$}
\label{sec: stone}
Throughout the present section we denote by $RG$ the \emph{commutative} $R$-algebra of $R$-valued functions on $G$ equipped with the pointwise product.
As an $R$-module, $RG$ is free with a basis given by the set of indicator functions $\delta_{g}$.
Under the pointwise product $\delta_{g} \delta_{k}$ is equal to $\delta_{g}$ if $g = k$, and is zero otherwise.
The group $G$ acts on $RG$ by $(g \cdot f) (k) = f(g^{-1}k)$, which is equivalent to $g \cdot \delta_k = \delta_{gk}$.

\begin{lem}
\label{lem: RG_B}
If $B$ is an $R$-algebra equipped with an action of $G$ then the map
\begin{align}
\label{eqn: RG_B}
B & \to \left( RG \otimes B \right)^{G} \\
\nonumber
b & \mapsto \sum_{g \in G} \delta_{g} \otimes (g \cdot b) 
\end{align}
is an $R$-algebra isomorphism.
\end{lem}

\begin{proof}
An element $\sum_{g} \delta_{g} \otimes b_{g}$ of $RG \otimes B$ is $G$-invariant if and only if
\[
\sum_g \delta_{kg} \otimes ( k \cdot b_g ) = \sum_g \delta_g \otimes b_g
\]
for all $k \in G$, which is equivalent to the condition that $b_k = k \cdot b_{1}$ for each $k \in G$; this shows that \eqref{eqn: RG_B} is an isomorphism.
The map \eqref{eqn: RG_B} is multiplicative because if $b,b' \in B$ then
\begin{align*}
\left( \sum_{g \in G} \delta_{g} \otimes (g \cdot b) \right) \left( \sum_{k \in G} \delta_{k} \otimes (k \cdot b') \right) & = \sum_{g,k \in G} \delta_{g} \delta_{k} \otimes (g \cdot b) (k \cdot b') \\
& = \sum_{g \in G} \delta_{g} \otimes g \cdot (bb') .
\qedhere
\end{align*}
\end{proof}

\begin{prop}[{c.f.\ \cite[Thm.\ 7.0.3]{PalmaBook2018}}]
\label{prop: stone}
For $\alpha : G \to \Aut_{R\text{-alg}}(RG)$ the group homomorphism determined by $(\alpha_{g} \phi)  (k) = \phi (g^{-1}k)$ for $g,k \in G$ and $\phi \in RG$, there is an isomorphism of $R$-algebras
\[
\mathcal{H}_{R}(G,H,RG,\alpha) \cong \End_{R} (R[G/H]) .
\]
\end{prop}

\begin{proof}
By Theorem \ref{thm: fixed_points} there is an $R$-algebra isomorphism
\[
\mathcal{H}_{R}(G,H,RG,\alpha) \cong \left( RG \otimes \End_{R} (R[G/H]) \right)^{G} 
\]
and by Lemma \ref{lem: RG_B} there is an $R$-algebra isomorphism
\[
\left( RG \otimes \End_{R} (R[G/H]) \right)^{G} \cong \End_{R} (R[G/H]) .
\qedhere
\]
\end{proof}

\begin{cor}
\label{cor: skew_hecke_stone}
The algebra $\mathcal{H}_{R}(G,H,RG,\alpha)$ is isomorphic to the algebra $M_{n}(R)$ of $n \times n$ matrices over $R$ for $n = |G/H|$.
In particular, $\mathcal{H}_{R}(G,H,RG,\alpha)$ is Morita equivalent to $R$, and is simple if $R$ is a field.
\end{cor}

\begin{remi}
\label{remi: palma_stone}
In the C*-algebra setting Palma has proven a generalisation of Proposition \ref{prop: stone} for general Hecke pairs $H \le G$ with $G$ possibly infinite, see \cite[Thm.\ 7.0.3]{PalmaBook2018} or \cite[Thm.\ 5.3]{Palma2013}.
In the setup and notation of \cite{PalmaBook2018}, Palma's `Stone-von Neumann theorem' is that
\[
C_0(G/\Gamma) \rtimes_{\alpha} G/\Gamma \cong C_0(G/\Gamma) \rtimes_{\alpha,r} G/\Gamma \cong \mathcal{K} (l^2(G/\Gamma))
\]
where the first two terms are particular C*-algebra completions of the skew Hecke algebra $\mathcal{H}_{\mathbb{C}} (G,\Gamma,C_0(G),\alpha)$, and $\mathcal{K}(l^2(G/\Gamma))$ is the C*-algebra of compact operators on the Hilbert space $l^2(G/\Gamma)$.
\end{remi}

\section{Group operations}
\label{sec: change_groups}
In this section we prove several results concerning the behaviour the skew Hecke algebra $\HGHA$ under certain group-theoretic operations: factoring out a normal subgroup $N \le G$ contained in $H$, replacing $G$ by a product group $G_1 \times G_2$, replacing $G$ by a subgroup $K \le G$ containing $H$, or replacing $H$ by a conjugate subgroup $H'$.
We also prove a result relating skew Hecke algebras to semi-direct products of groups.

\subsection{Factoring out normal subgroups}
\label{sec: factoring_out_normal_subgroups}
The following result is an analogue of the corresponding result \cite[Chap.\ 1, Cor.\ 6.2]{KriegBook1990} for Hecke algebras, and is a generalisation of Proposition \ref{prop: special_cases} \eqref{it: spec_case.N}.

\begin{prop}[c.f.\ {\cite[Chap.\ 1, Cor.\ 6.2]{KriegBook1990}}]
\label{prop: factoring_out_normal}
If $N \le G$ is a normal subgroup and $N \subseteq H$ then there is an isomorphism
\[
\mathcal{H}_{R} (G,H,A,\alpha) \cong \mathcal{H}_{R} (G/N,H/N,A^{N},\alpha^{N}) 
\]
where $\alpha^{N} : G/N \to \Aut_{R\text{-alg}}(A)$ is the homomorphism $gN \mapsto \alpha_{g}$.
\end{prop}

\begin{proof}
The subgroup $N \le G$ acts trivially on $\IndHG R$ because $N$ is normal and $N \le H$, and as an $R[G/N]$-module $\IndHG R \cong \Ind_{H/N}^{G/N} R$.
Then
\[
( A \otimes \End_R(\IndHG) )^{G} =
( A^{N} \otimes \End_R(\IndHG) )^{G/N} =
( A^{N} \otimes \End_R(\Ind_{H/N}^{G/N} R) )^{G/N}
\]
which is isomorphic to $\mathcal{H}_{R}(G/N,H/N,A,\alpha^{N})$.
\end{proof}

\subsection{Product groups}
\label{sec: product_groups}
The following result is an analogue of the corresponding result for Hecke algebras \cite[Chap.\ 1, Thm.\ 6.3]{KriegBook1990}.

\begin{prop}[c.f.\ {\cite[Chap.\ 1, Thm.\ 6.3]{KriegBook1990}}]
\label{prop: products}
Suppose that $(G_1,H_1,A_1,\alpha_1)$ and $(G_2,H_2,A_2,\alpha_2)$ are tuples as in \S \ref{sec: setup_tuples}, and $\alpha_1 \otimes \alpha_2$ is the group homomorphism
\begin{align*}
\alpha_1 \otimes \alpha_2 : G_1 \times G_2 & \to \Aut_{R\text{-alg}}(A_1 \otimes A_2) \\
(g_1,g_2) & \mapsto \alpha_1(g_1) \otimes \alpha_2(g_2) .
\end{align*}
If either $A_1$ and $A_2$ are free as $R$-modules, or $|H_1|$ and $|H_2|$ are units in $R$, then the skew Hecke algebra
\[
\mathcal{H}_{R}(G_1 \times G_2,H_1 \times H_2,A_1 \otimes A_2,\alpha_1 \otimes \alpha_2) 
\]
is isomorphic to the tensor product of $R$-algebras
\[
\mathcal{H}(G_1,H_1,A_1,\alpha_1) \otimes \mathcal{H}(G_2,H_2,A_2,\alpha_2) .
\]
\end{prop}

\begin{proof}
As an $R[G_1 \times G_2]$-module $\Ind_{H_1 \times H_2}^{G_1 \times G_2}R \cong \Ind_{H_1}^{G_1}R \otimes \Ind_{H_2}^{G_2}R$.
If $V_1$ (resp.\ $V_2$) is an $R[G_1]$ (resp.\ $R[G_2]$)-module, and $V_1,V_2$ are free over $R$, then $(V_1 \otimes V_2)^{G_1 \times G_2} = V_1^{G_1} \otimes V_2^{G_2}$.
It follows that if $A_{1}$ and $A_{2}$ are free over $R$ then by reordering tensor factors we get an isomorphism
\begin{align*}
& ( A_{1} \otimes \End_{R}(\Ind_{H_1}^{G_1}R) )^{G_1} \otimes ( A_{2} \otimes \End_{R}(\Ind_{H_2}^{G_2}R) )^{G_2} \\
& \cong (A_{1} \otimes A_{2} \otimes \End_{R} (\Ind_{H_1 \times H_2}^{G_1 \times G_2}))^{G_1 \times G_2} .
\end{align*}
If $|H_1|$ and $|H_2|$ are units in $R$ then the natural isomorphism
\[
(A_1 \otimes A_2) \rtimes_{\alpha_1 \otimes \alpha_2} (G_1 \times G_2) \cong \left( A_1 \rtimes_{\alpha_1} G_1 \right) \otimes \left( A_2 \rtimes_{\alpha_2} G_2 \right)
\]
given by reordering tensor factors maps the idempotent $\mathbf{e}_{H_1 \times H_2}$ to the tensor product of the idempotents $\mathbf{e}_{H_1}$ and $\mathbf{e}_{H_2}$, and 
\begin{align*}
& (\mathbf{e}_{H_1} \otimes \mathbf{e}_{H_2}) \left( \left( A_1 \rtimes_{\alpha_1} G_1 \right) \otimes \left( A_2 \rtimes_{\alpha_2} G_2 \right) \right) (\mathbf{e}_{H_1} \otimes \mathbf{e}_{H_2}) \\
& \cong \left( \mathbf{e}_{H_1} \left(A_{1} \rtimes_{\alpha_1} G_{1}\right) \mathbf{e}_{H_1} \right)  \otimes \left( \mathbf{e}_{H_2} \left(A_{2} \rtimes_{\alpha_2} G_{2}\right) \mathbf{e}_{H_2} \right) .
\end{align*}
\end{proof}


\subsection{Intermediate subgroups}
\label{sec: skew_hecke_intermediate_subgroups}
The following result is an analogue of the statement that if $K \le G$ is a subgroup then $A \rtimes K$ is a subalgebra of $A \rtimes G$.

\begin{prop}
\label{prop: intermediate_subgroups}
Suppose that $K \le G$ is a subgroup of $G$ such that $H \subseteq K \subseteq G$.
Extending functions by zero defines an injective $R$-algebra morphism
\[
\mathcal{H}_{R}(K,H,A,\alpha) \hookrightarrow \HGHA
\]
where we denote by the same symbol $\alpha$ the restriction of $\alpha$ to $K$.
\end{prop}

\begin{proof}
There is an injective $H$-equivariant map 
\[
K/H \to G/H \; , \; kH \mapsto gH
\]
so that an element $\sum_{kH \in K/H} a_{kH} \otimes kH$ of $(A \otimes R[K/H])^{H}$ can be considered as an element of $(A \otimes R[G/H])^{H}$.
The product of two such elements in $\mathcal{H}_{R}(K,H,A,\alpha)$ agrees with their product in $\HGHA$, because if $kH,k'H \in K/H$ then $kk'H \in K/H$ also.
\end{proof}

\begin{remi}
Taking $K = H$ in Proposition \ref{prop: intermediate_subgroups} corresponds to the natural morphism from $A^{H}$ to $\HGHA$, see \S \ref{sec: two_subalgebras}.
\end{remi}


\subsection{Conjugate subgroups}
\label{sec: skew_hecke_fun_conjugation}
The following result is a generalisation of the fact that if $H,H' \le G$ are conjugate subgroups of $G$, then $\HGH \cong \mathcal{H}_{R}(G,H')$.

\begin{prop}
\label{prop: skew_hecke_fun_conjugation}
If $H,H' \le G$ are conjugate subgroups of $G$ then there is an isomorphism of $R$-algebras
\[
\HGHA \cong \mathcal{H}_{R}(G,H',A,\alpha) .
\]
\end{prop}

\begin{proof}
If $H$ is conjugate to $H'$ then $\IndHG R \cong \Ind_{H'}^{G} R$ as $R[G]$-modules.
It follows that
\[
\left( A \otimes \End_{R} ( \IndHG R ) \right)^{G} \cong
\left( A \otimes \End_{R} ( \Ind_{H'}^{G} R ) \right)^{G} 
\]
as $R$-algebras, whence $\HGHA \cong \mathcal{H}_{R}(G,H',A,\alpha)$ by Theorem \ref{thm: fixed_points}.
\end{proof}


\subsection{Semi-direct products}
\label{sec: semi_direct}
\begin{prop}
\label{prop: skew_hecke_semi_direct}
Suppose that $N \rtimes K$ is a semi-direct product group with $K$ finite, and $H \le K$ is a subgroup.
Let $\alpha$ be the $R$-linear extension to $R[N]$ of the action of $K$ on $N$.
There is an $R$-algebra isomorphism
\[
\mathcal{H}_{R}(K,H,R[N],\alpha) \cong \mathcal{H}_{R}(N \rtimes K,H) .
\]
\end{prop}

\begin{proof}
As $R$-algebras $R[N \rtimes K] \cong R[N] \rtimes K$, and under this isomorphism the $R[N \rtimes K]$-module $R[ (N \rtimes K) / H ]$ corresponds to the $R[N] \rtimes K$-module $(R[N]) [K/H]$.
Then
\[
\End_{R[N] \rtimes K} ( R[N][K/H] )
\cong
\End_{R[N \rtimes K]} (R[(N \rtimes K)/H]) .
\]
\end{proof}



\section{Further properties}
\label{sec: further_properties}
In this section we show that the construction of skew Hecke algebras is compatible with restriction and extension of the ground ring and with localisation at a prime ideal $\mathfrak{p} \subset R$, is stable under certain cocycle perturbations of the action $\alpha$, is compatible with gradings and filtrations, and is compatible with the formation of opposite algebras.

\subsection{Restriction and extension of scalars}
\label{sec: ext_scalars}
Let $R$ be a commutative ring and $(G,H,A,\alpha)$ be as in \S \ref{sec: setup_tuples}.
Suppose that $R'$ is a commutative ring and $R$ is an $R'$-algebra.
By restriction, $A$ is an $R'$-algebra, the action $\alpha$ is $R'$-linear, and
\[
\mathcal{H}_{R'}(G,H,A,\alpha) = \mathcal{H}_{R}(G,H,A,\alpha)
\]
as $R'$-algebras.

\begin{prop}
\label{prop: ext_scalars}
Suppose that $S$ is a commutative $R$-algebra.
Let $\alpha_{S}$ be group homomorphism
\[
\alpha_S : G \to \Aut_{S\text{\emph{-alg}}}(S \otimes_{R} A) \; , \; \alpha_S(g) = \mathrm{id}_S \otimes \alpha(g) .
\]
If either $S$ is a free $R$-module, or $|H|$ is a unit in $A$ then there is an isomorphism of $S$-algebras
\[
S \otimes_{R} \HGHA \cong \mathcal{H}_{S}(G,H,S \otimes_{R} A,\alpha_S) .
\]
\end{prop}

\begin{proof}
This follows from Proposition \ref{prop: products} by noting that $S \cong \mathcal{H}_{R}(1,1,S,\mathrm{triv})$.
Note that in the proof of Proposition \ref{prop: products}, if $G_1 = 1$ then it is sufficient to assume that $S$ is free over $R$ to ensure that $(S \otimes V)^{G} = S \otimes V^{G}$ for any $R[G]$-module $V$.
\end{proof}


\subsection{Families}
\label{sec: families}
Let $\HGHA$ be a skew Hecke algebra over $R$.
For each prime ideal $\mathfrak{p} \subset R$, with residue field $\kappa(\mathfrak{p}) = R_{\mathfrak{p}} / \mathfrak{p}R_{\mathfrak{p}}$, there is an associated $\kappa(\mathfrak{p})$-algebra
\[
\HGHA^{\#} _{\mathfrak{p}} := \kappa(\mathfrak{p}) \otimes_{R} \HGHA ,
\]
which is equal to the fibre over $\mathfrak{p}$ of the sheaf of $\mathcal{O}_{\mathrm{Spec} \, R}$-algebras on $\mathrm{Spec} \, R$ associated to the $R$-algebra $\HGHA$.

\begin{prop}
If $|H|$ is a unit in $R$ then for each $\mathfrak{p} \in \mathrm{Spec} \, R$ there is an isomorphism of $\kappa(\mathfrak{p})$-algebras
\[
\HGHA^{\#} _{\mathfrak{p}} \cong \mathcal{H}_{\kappa(\mathfrak{p})} (G,H,\kappa(\mathfrak{p}) \otimes_{R} A,\alpha_{\kappa(\mathfrak{p})}) .
\]
\end{prop}

\begin{proof}
This follows immediately from Proposition \ref{prop: ext_scalars}.
\end{proof}

\begin{ex}
Suppose that $A$ is an algebra over an algebraically closed field $k$, $G$ is a group, $H \le G$ is a subgroup, $\alpha : G \to \Aut_{k\text{-alg}} (A)$ is an action of $G$ on $A$, and $\xi \in Z(A)^{G}$ is a $G$-invariant central element.
The assignment $t \mapsto \xi$ extends uniquely to a $k[t]$-algebra structure on $A$, compatible with the action of $G$, and we have the skew Hecke algebra $\mathcal{H}_{k[t]}(G,H,A,\alpha)$ over $k[t]$.
The spectrum of $k[t]$ is
\[
\mathrm{Spec} \, k[t] = \{ (0) \} \cup \{ (t-\lambda) \;|\; \lambda \in k\}
\] 
and if $|H|$ is not divisible by the characteristic of $k$ then for any $\lambda \in k$
\[
\mathcal{H}_{k[t]}(G,H,A,\alpha) / (\xi - \lambda 1_{A}) \cong \mathcal{H}_{k}\left(G,H,A/(\xi-\lambda 1_A),\alpha_{k[t]/(t-\lambda)}\right) .
\]
\end{ex}

\subsection{Cocycle-equivalent actions}
\label{sec: coycle_equiv}

\begin{prop}
\label{prop: cocycle_equivalence}
Suppose that $\alpha , \beta : G \to \Aut_{R\text{-alg}}(A)$ are group homomorphisms, and that there exists a map $\chi : G \to A^{\times}$ from $G$ to the group of units in $A$, such that:
\begin{enumerate}
\item \label{it: cocycle_equivalence.chi}
$\chi(gg') = \chi(g) \alpha_{g} \chi(g')$ for all $g,g' \in G$.

\item \label{it: cocycle_equivalence.inner}
$\beta_{g} a = \chi(g) (\alpha_{g} a) \chi(g)^{-1}$ for all $g \in G$ and $a \in A$.

\item \label{it: cocycle_equivalence.h}
$\chi(h) = 1_{A}$ for all $h \in H$.
\end{enumerate} 
Then there is an isomorphism of $R$-algebras
\begin{equation}
\label{eqn:isom_cocyle_equivalent}
\HGHA \cong \mathcal{H}_{R}(G,H,A,\beta) .
\end{equation}
\end{prop}

\begin{proof}
Set $E := \End_{R}(\IndHG R)$; note that $E \cong E^{\mathrm{op}}$.
Write $\mathrm{Ad}_{g}$ for the action of $G$ on $E$.
We use the isomorphism $\HGHA \cong (A \otimes E)^{G}$ of Theorem \ref{thm: fixed_points}, and the notation $E_{gH,kH}$ used in the proof of that result.
There is an $R$-algebra automorphism $\Phi$ of $A \otimes E$ determined by
\[
a \otimes E_{gH,kH} \mapsto \chi(g) a \chi(k)^{-1} \otimes E_{gH,kH} .
\]
The map is well defined by assumption \eqref{it: cocycle_equivalence.h}.
The fact that $\Phi$ is multiplicative follows from the fact that $E_{gH,kH} E_{sH,tH} = 0$ unless $kH = sH$, in which case for any $a,b \in A$
\begin{align*}
& \left( \chi(g) a \chi(k)^{-1} \otimes E_{gH,kH} \right) \left( \chi(s) b \chi(t)^{-1} \otimes E_{sH,tH} \right) \\
& = \left( \chi(g) a \chi(k)^{-1} \chi(k) b \chi(t)^{-1} \otimes E_{gH,tH} \right) \\
& = \chi(g) a b \chi(t)^{-1} \otimes E_{gH,tH} .
\end{align*}
It follows from the cocycle condition of assumption \eqref{it: cocycle_equivalence.chi} that
\begin{align*}
\chi(sg) (\alpha_{s}a) \chi(sk)^{-1} & = \chi(s)\alpha_{s}(\chi(g)) \chi(s)^{-1} (\beta_{s}a) \chi(s) \alpha_{s}(\chi(k))^{-1} \chi(s)^{-1} \\
& = \left( \beta_{s} \chi(g) \right) (\beta_{s}a) \left( \beta_{s}\chi(k)^{-1} \right) ,
\end{align*}
which gives the third equality in 
\begin{align*}
\left( \Phi \circ ( \alpha_{s} a\otimes \mathrm{Ad}_{s} ) \right) (a \otimes E_{gH,kH}) & = \Phi (\alpha_{s} a \otimes E_{sgH,skH}) \\
& = \chi(sg) (\alpha_{s}a) \chi(sk)^{-1} \otimes E_{sgH,skH} \\
& = \left( \beta_{s} \chi(g) \right) (\beta_{s}a) \left( \beta_{s}\chi(k)^{-1} \right) \otimes E_{sgH,skH} \\
& = (\beta_{s} \otimes \mathrm{Ad}_{s}) ( \chi(g) a \chi(k)^{-1} \otimes E_{gH,kH} ) \\
& = \left( (\beta_{s} \otimes \mathrm{Ad}_{s}) \circ \Phi \right) ( a \otimes E_{gH,kH} ) .
\end{align*}
This shows that $\Phi$ maps the $\alpha \otimes \mathrm{Ad}$-invariants isomorphically onto the $\beta \otimes \mathrm{Ad}$-invariants, and the result follow from Theorem \ref{thm: fixed_points}.
\end{proof}

\begin{remi}
In terms of $A[G/H]^{H}$, an isomorphism \eqref{eqn:isom_cocyle_equivalent} is given by
\[
\sum_{gH} a_{gH} gH \mapsto \sum_{gH} (a_{gH} \chi(g)^{-1}) gH .
\]
\end{remi}

\begin{cor}
\label{cor: inner_action}
If there exists a homomorphism $\chi : G \to A^{\times}$ such $\alpha_{g}(a) = \chi(g)a\chi(g)^{-1}$ and $\chi(h) = 1_{A}$ for all $g \in G$, $a \in A$ and $h \in H$; and $A$ is free as an $R$-module, then there is an isomorphism of $R$-algebras
\[
\HGHA \cong A \otimes \HGH .
\]
\end{cor}

\begin{proof}
This follows from Proposition \ref{prop: special_cases}\eqref{it: spec_case.a_triv} and Proposition \ref{prop: cocycle_equivalence} applied to the case where $\alpha$ is the trivial homomorphism.
\end{proof}

\begin{remi}
\label{rem: outer_equiv_idempotent}
Proposition \ref{prop: cocycle_equivalence} and Corollary \ref{cor: inner_action} are in general false if the assumption that $\chi(h) = 1_{A}$ for all $h \in H$ is dropped.
By Proposition \ref{prop: special_cases} \eqref{it: spec_case.H_eq_G}, $\mathcal{H}_{R}(G,G,A,\alpha) \cong A^{G}$ whether or not the action $\alpha$ is inner, whereas $A \otimes \mathcal{H}_{R}(G,G) \cong A$.
\end{remi}

\subsection{Gradings and filtrations}
\label{sec: grad_filt}
Set $E := \End_{R}(\IndHG R)$.
If $A$ is equipped with an $R$-algebra grading $A = \bigoplus_{i \in \mathbb{Z}} A_{i}$ such that each $R$-submodule $A_{i} \subseteq A$ is an $R[G]$-submodule, then $(A \otimes E)^{G} \cong \HGHA$ is a graded $R$-algebra with $((A \otimes E)^{G} )_{i} = (A_{i} \otimes E)^{G}$.
Similarly, if $A$ is equipped with a $G$-invariant $R$-algebra filtration $\{A_{\le i}\}_{i \in \mathbb{Z}}$ then $\HGHA$ is a graded $R$-algebra with $((A \otimes E)^{G})_{\le i} = (A_{\le i} \otimes E)^{G}$.
In the latter situation, there is a natural action $\alpha^{\mathrm{Gr}}$ of $G$ on the associated graded algebra $\Gr A := \bigoplus_{i \in \mathbb{Z}} A_{\le i}/A_{\le i-1}$ and we can consider the corresponding skew Hecke algebra.

\begin{prop}
\label{prop: skew_hecke_grading}
Suppose that $A$ is equipped with a $G$-invariant $R$-algebra filtration.
There is a homomorphism of graded $R$-algebras
\[
\Gr \HGHA \to \mathcal{H}_{R}(G,H,\Gr A,\alpha^{\mathrm{Gr}})
\]
that is an isomorphism if $|G|$ is a unit in $R$.
\end{prop}

\begin{proof}
The required $R$-algebra homomorphism is given by the direct sum of the natural maps
\[
(A_{\le i} \otimes E)^{G} / (A_{\le i-1} \otimes E)^{G} \to (A_{\le i}/A_{\le i-1} \otimes E)^{G} = ((\Gr A \otimes E)^{G})_{i} .
\]
If $|G|$ is a unit in $R$ then the functor $(- \otimes E)^{G}$ on $R$-modules is exact because $E$ is free and $-^{G}$ is exact, and the result follows from the short exact sequence
\[
0 \to A_{\le i-1} \to A_{\le i} \to A_{\le i} / A_{\le i-1} \to 0 .
\]
\end{proof}

\begin{remi}
With more work one can show that it is sufficient to assume that $|H|$ is a unit in $R$ by considering $A[G/H]^{H}$ instead of $(A \otimes E)^{G}$.
\end{remi}

\subsection{Opposite algebras and involutions}
\label{sec: opposite_involutions}

\begin{prop}
\label{prop: op_alg_isom}
There is an $R$-algebra isomorphism
\[
\HGHA^{\mathrm{op}} \cong \mathcal{H}_{R}(G,H,A^{\mathrm{op}},\alpha^{\mathrm{op}})
\]
where $\alpha_{g}^{\mathrm{op}}a := \alpha_{g}a$ for $a \in A^{\mathrm{op}}$.
\end{prop}

\begin{proof}
Set $E := \End_{R} (\IndHG R)^{\mathrm{op}}$.
Then as $R$-algebras and $R[G]$-modules $E^{\mathrm{op}} \cong E$ via the transpose map, and so
\[
(A \otimes E)^{\mathrm{op}} \cong A^{\mathrm{op}} \otimes E^{\mathrm{op}} \cong A^{\mathrm{op}} \otimes E .
\]
The statement then follows by taking $G$-invariants and using the isomorphism $\HGHA \cong (A \otimes E)^{G}$ of Theorem \ref{thm: fixed_points}.
\end{proof}

\begin{remi}
An explicit $R$-algebra isomorphism $(A[G/H]^{H})^{\mathrm{op}} \cong A^{\mathrm{op}}[G/H]^{H}$ is given by
\[
\sum_{gH} a_{gH} gH \mapsto \sum_{gH} (\alpha_{g} a_{g^{-1}H}^{\mathrm{op}}) gH
\]
where $a^{\mathrm{op}}$ denotes an element $a \in A$ considered as an element of $A^{\mathrm{op}}$.
\end{remi}

\begin{cor}
\label{cor: isom_to_op}
The following statements hold.
\begin{enumerate}
\item If $A \cong A^{\mathrm{op}}$ as $R$-algebras equipped with an action of $G$ then
\[
\HGHA \cong \HGHA^{\mathrm{op}} .
\]

\item If $A$ is commutative then $\HGHA$ is isomorphic to its opposite algebra $\HGHA^{\mathrm{op}}$.

\item If $A$ is free as $R$-module, or if $|H|$ is a unit in $R$, then the enveloping algebra
\[
\HGHA \otimes \HGHA^{\mathrm{op}}
\]
is isomorphic to the skew Hecke algebra
\[
\mathcal{H}_{R} ( G \times G , H \times H , A \otimes A^{\mathrm{op}} , \alpha \otimes \alpha^{\mathrm{op}} ) .
\]
\end{enumerate}
\end{cor}

\begin{proof}
These statements follow immediately from Proposition \ref{prop: op_alg_isom} and Proposition \ref{prop: products}.
\end{proof}

\begin{remi}
A related result is proven in \cite[Prop.\ 3.14]{PalmaBook2018}, where it is shown that a $G$-invariant involution on $A$ determines an involution on $\HGHA$.
\end{remi}


\section{Example: $S_2 \le S_3$}
\label{sec: example_S3}
In this section we choose a particular example of a skew Hecke algebra and calculate the product, the double coset decomposition given by Theorem \ref{thm: double_coset}, and the isomorphism onto a fixed point algebra given by Theorem \ref{thm: fixed_points}.
\\
\\
Given data $(G,H,A,\alpha)$, Proposition \ref{prop: special_cases} shows that if $H$ is normal, or $G$ acts trivially on $A$, then the skew Hecke algebra $\HGHA$ is either an invariant ring, a skew group algebra, a (non-skew) Hecke algebra, or a tensor product of such algebras.
In this section we study $\HGHA$ for the case that $H$ is the smallest example of a non-normal subgroup, which is $S_2 \le S_3$.
\\
\\
For $G = S_3$, the three subgroups of order 2, say $S_2,S_2',S_2''$, are conjugate, and the alternating subgroup $A_3 \le S_3$ is normal.
It then follows from Proposition \ref{prop: skew_hecke_fun_conjugation} that there are isomorphisms
\[
\mathcal{H}_{R}(S_3,S_2,A,\alpha) \cong \mathcal{H}_{R}(S_3,S_2',A,\alpha) \cong \mathcal{H}_{R}(S_3,S_2'',A,\alpha) ,
\]
and from Proposition \ref{prop: special_cases} that there are isomorphisms
\begin{align*}
\mathcal{H}_{R}(S_3,A_3,A,\alpha) & \cong A^{A_3} \rtimes S_3/A_3 \\
\mathcal{H}_{R}(S_3,S_3,A,\alpha) & \cong A^{S_3} \\
\mathcal{H}_{R}(S_3,1,A,\alpha) & \cong A .
\end{align*}
We therefore focus on the case that $H \le S_3$ is a particular subgroup of order 2.
In \S \ref{sec: example_S3_general} $A$ is any $R$-algebra equipped with an action $\alpha$ of $S_3$, in \S \ref{sec: polynomial_rings} we specialise to the case that $A = R[x_1,x_2,x_3]$, in \S \ref{sec: example_functions_on_S3} $A$ is the commutative algebra $RS_3$ of functions on $S_3$, and in \S \ref{sec: example_semidirect} $A$ is the group algebra $R[K]$ for $K$ a finite group.

\subsection{General $A$}
\label{sec: example_S3_general}
Throughout \S \ref{sec: example_S3_general}, $A$ is an arbitrary $R$-algebra equipped with an action $\alpha$ of $S_3$ by $R$-algebra automorphisms.

\subsubsection{The group $S_3$}
\label{sec: example_S3_setup}
We take $S_3$ to be the group of permutations of the set $\{1,2,3\}$, and $S_2$ to be stabiliser of the element $3$, so that
\begin{align*}
S_3 & = \{ \mathrm{id} , (12) , (23) , (13) , (123) , (132) \} \\
S_2 & = \{ \mathrm{id} , (12) \} 
\end{align*}
where $\mathrm{id}$ is the identity map.
The left cosets of $S_2$ are
\[
S_2 = \{ \mathrm{id} , (12) \} \; , \; (23) S_2 = \{ (23) , (132) \} \; , \; (13) S_2 = \{ (13) , (123) \} .
\]
We fix the representatives $g_1 = \mathrm{id}$, $g_2 = (23)$, $g_3 = (13)$ of these cosets, and label the cosets 
\[
C_1 := S_2 \; , \; C_2 := (23)S_2 \; , \; C_3 := (13)S_2 .
\]
We fix the representatives $\mathrm{id}$ and $(23)$ of the double cosets of $S_2$, which are
\[
S_2 \; , \; S_2 (23) S_2  = \{ (13) , (23) , (123) , (132) \} .
\]

\subsubsection{The $R$-module structure}
\label{sec: example_S3_Rmodule}
As an $R$-module the skew Hecke algebra $\HGHASym$ is equal to
\begin{align*}
\Maps ( S_3 / S_2 , A )^{S_2} & = \{ \phi : S_3 / S_2 \mid \phi( (12) gS_2 ) = \alpha_{(12)} \phi(gS_2) \text{ for all } g \in S_3 \} \\
& \cong \left( A \otimes R[S_3/S_2] \right)^{S_2} .
\end{align*}
It follows from the identity $(12)(23)S_2 = (13)S_2$ that a map $\phi : S_3/S_2 \to A$ is in $\Maps(S_3/S_2,A)^{S_2}$ if and only if
\[
\phi( S_2 ) \in A^{S_2} \; , \; \phi( (13)S_2 ) = \alpha_{(12)} \phi( (23) S_2 ) .
\]
With respect to the double coset representatives $\mathrm{id}$ and $(23)$, the $A^{H}$-bimodule isomorphism of Theorem \ref{thm: double_coset}
\[
\HGHASym \xrightarrow{\cong} \sum_{S_2 g S_2} A^{S_2 \cap gS_2g^{-1}}
\]
is
\begin{equation}
\label{eqn: example_S3_R_module_isom}
\HGHASym \xrightarrow{\cong} A \oplus A^{S_2} \; , \; \phi \mapsto \left( \phi(S_2) , \phi( (23)S_2) \right) .
\end{equation}
The inverse of this isomorphism is the map
\[
(a,b) \mapsto \left( S_2 \mapsto a \; , \; (12)S_2 \mapsto b \; , \; (13)S_2 \mapsto \alpha_{(12)} b \right) .
\]
When comparing with Theorem \ref{thm: double_coset} note that $S_2 \cap (23) S_2 (23)^{-1} = \{ \mathrm{id} \}$.
In terms of tensors, a general element of $\HGHASym \cong (A \otimes R[S_3/S_2])^{S_2}$ is of the form
\[
a \otimes S_2 + b \otimes (23)S_2 + \alpha_{(12)} b \otimes (23)S_2
\]
or
\[
a \otimes C_1 + b \otimes C_2 + \alpha_{(12)} b \otimes C_3 .
\]

\subsubsection{The product}
\begin{prop}
\label{prop: example_S3_skew_hecke_product}
Under the $R$-module isomorphism \eqref{eqn: example_S3_R_module_isom}, the product on the $R$-algebra $\HGHASym$ corresponds to the product on $A^{S_2} \oplus A$ given by the formula
\begin{equation}
\label{eqn: example_S3_skew_hecke_product}
\begin{pmatrix}
a \\
b
\end{pmatrix}
\begin{pmatrix}
a' \\
b'
\end{pmatrix}
 = 
\begin{pmatrix}
aa' + \left( \mathrm{id} + \alpha_{(12)} \right) ( b \alpha_{(23)} b' ) \\
ab' + b \alpha_{(23)} a' + (\alpha_{(12)} b) (\alpha_{(13)} b')
\end{pmatrix}
\end{equation}
where $\mathrm{id} + \alpha_{(12)}$ is the sum in $\End_{R} A$ of the identity endomorphism and the automorphism $\alpha_{(12)}$.
Equivalently, the direct summand $A^{S_2}$ is a subalgebra with product restricted from $A$; the direct summand $A$ is an $A^{S_2}$-bimodule with actions
\begin{equation}
\label{eqn: example_S3_bimodule_A}
a \cdot b \cdot a' = a b (\alpha_{(23)} a')
\end{equation}
for $a,a' \in A^{S_2}$ and $b \in A$; and the product of two elements $b,b'$ in the direct summand $A$ is given by
\begin{equation}
\label{eqn: example_S3_product_bb'}
\begin{pmatrix}
\left( \mathrm{id} + \alpha_{(12)} \right) ( b \alpha_{(23)} b' ) \\
(\alpha_{(12)} b) (\alpha_{(13)} b') .
\end{pmatrix}
\end{equation}
\end{prop}

\begin{proof}
It follows from \S \ref{sec: R_module_structure} and Theorem \ref{thm: double_coset} that the summand $A^{S_{2}}$ is a subalgebra with product equal to that restricted from $A$, and the summand $A$ is an $A^{H}$-bimodule with actions given by \eqref{eqn: example_S3_bimodule_A}.
It remains to calculate the product of two elements in the summand $A$.
Suppose that $b,b' \in A$ and $\phi,\phi' \in \HGHASym$ are the corresponding $S_{2}$-invariant functions defined by
\[
\phi(S_2) = \phi'(S_2) = 0
\]
and
\begin{align*}
\phi( (23)S_2 ) = b \; , & \; \phi( (13)S_2 ) = \alpha_{(13)} b \\
\phi'( (23)S_2 ) = b' \; , & \; \phi'( (13)S_2 ) = \alpha_{(13)} b' .
\end{align*}
Using the formula \eqref{eqn: skew_hecke_product} for the convolution product on $\HGHASym$ gives
\begin{align*}
( \phi * \phi' ) ( S_2 ) & = \phi(S_2) \phi'(S_2) + \phi( (23)S_2 ) \alpha_{(23)} \phi'( (23)S_2 ) + \phi( (13) S_2 ) \alpha_{(13)} \phi'( (13)S_2 ) \\
& = b \alpha_{(23)} b' + (\alpha_{(12)} b) \alpha_{(13)} (\alpha_{(12)} b') \\
& = b \alpha_{(23)} b' + (\alpha_{(12)} b) \alpha_{(12)} (\alpha_{(23)} b') \\
& = b \alpha_{(23)} b' + \alpha_{(12)} ( b \alpha_{(23)} b' ) \\
& = \left( \mathrm{id} + \alpha_{(12)} \right) ( b \alpha_{(23)} b' )
\end{align*}
where $\mathrm{id} + \alpha_{(12)}$ is the sum in $\End_{R} A$ of the identity endomorphism and the automorphism $\alpha_{(12)}$, and we used the identity $(13)(12) = (12)(23)$ in $S_3$.
Similarly,
\begin{align*}
( \phi * \psi ) ( (23)S_2 ) & = \phi( S_2 ) \phi'( (23)S_2 ) + \phi( (23)S_2 ) \alpha_{(23)} \phi'( S_2 ) + \phi( (13)S_2 ) \alpha_{(13)} \phi'( (23)S_2 ) \\
& = ( \alpha_{(12)} b ) ( \alpha_{(13)} b' ) .
\end{align*}
These computations show that the product $(0,b) (0,b')$ is equal to \eqref{eqn: example_S3_product_bb'}.
\end{proof}

\begin{remi}
If $A = R$ and $\alpha$ is the trivial action then $\HGHASym = \HGHSym$, and the formula \eqref{eqn: example_S3_skew_hecke_product} reduces to
\[
\begin{pmatrix}
r \\
s
\end{pmatrix}
\begin{pmatrix}
r' \\
s'
\end{pmatrix}
 = 
\begin{pmatrix}
rr' +  2ss' \\
rs' + sr' + ss'
\end{pmatrix}
\]
which agrees with the formula \cite[Lem.\ 5.1]{KriegBook1990} for the product in $\HGHSym$.
\end{remi}

\subsubsection{Fixed point algebras}
By Theorem \ref{thm: fixed_points} there is an $R$-algebra isomorphism
\begin{align}
\label{eqn: fixed_point_isom_S3}
\HGHASym & \xrightarrow{\cong} \left( A \otimes \End_{R} (R[S_{3}/S_{2}]) \right)^{S_3} \\
\nonumber
\phi & \mapsto \sum_{gS_2,kS_2 \in S_3/S_2} \alpha_{k} \phi(k^{-1}gS_2) \otimes E_{gS_2,kS_2} .
\end{align}
The labelling $C_1 = S_2$, $C_2 = (23)S_2$, $C_3 = (13)S_2$ identifies $\End_{R}(R[G/H])$ with $M_{n}(R)$, and $A \otimes \End_{R} (R[S_{3}/S_{2}])$ with $M_{n}(A)$.
Set $g_{1} := 1$, $g_{2} := (23)$, $g_{3} := (13)$.
Then the isomorphism \eqref{eqn: fixed_point_isom_S3} corresponds to the map
\begin{align}
\label{eqn: fixed_point_isom_S3_matrix}
\HGHASym & \xrightarrow{\cong} M_{3}(A)^{S_3} \\
\nonumber
\phi & \mapsto \left( \alpha_{g_j} \phi( g_{j}^{-1}C_{i} ) \right)_{1 \le i,j \le 3}
\end{align}
If $\phi \in \HGHASym$ and $a := \phi(C_1)$, $b := \phi(C_2)$, $c := \phi(C_3)$; then by \S \ref{sec: example_S3_Rmodule} we have $c = \alpha_{(12)} b$, and
the isomorphism \eqref{eqn: fixed_point_isom_S3_matrix} maps $\phi$ to the matrix
\begin{align*}
\begin{pmatrix}
\phi(C_1) & \alpha_{(23)}\phi((23)C_1) & \alpha_{(13)} \phi((13)C_1) \\ 
\phi(C_2) & \alpha_{(23)}\phi((23)C_2) & \alpha_{(13)} \phi((13)C_2) \\ 
\phi(C_3) & \alpha_{(23)}\phi((23)C_3) & \alpha_{(13)} \phi((13)C_3)
\end{pmatrix}
& = 
\begin{pmatrix}
a & \alpha_{(23)} b & \alpha_{(13)} c \\
b & \alpha_{(23)} a & \alpha_{(13)} b \\
c & \alpha_{(23)} c & \alpha_{(13)} a
\end{pmatrix}  
\\
& = 
\begin{pmatrix}
a & \alpha_{(23)} b & \alpha_{(13)(12)} b \\
b & \alpha_{(23)} a & \alpha_{(13)} b \\
c & \alpha_{(23)(12)} b & \alpha_{(13)} a
\end{pmatrix} 
\end{align*}
and the restriction of \eqref{eqn: fixed_point_isom_S3_matrix} to the subalgebra $A^{S_2} \subseteq \HGHASym$ is
\[
d \mapsto \begin{pmatrix}
d & 0 & 0 \\
0 & \alpha_{(23)} d & 0 \\
0 & 0 & \alpha_{(13)} d
\end{pmatrix} .
\]

\subsubsection{Corner rings}
If $|S_2| = 2$ is a unit in $A$, e.g.\ if $R$ is a field or characteristic not equal to $2$, then
\[
\mathbf{e}_{S_2} = \frac{1}{2} \sum_{g \in S_2} 1_{A} \otimes g = \frac{1}{2} ( 1_{A} \otimes 1_{G} + 1_{A} \otimes (12) )
\]
is an idempotent in the skew group algebra $A \rtimes S_3$, and by Theorem \ref{prop:skew_hecke_corner} there is an isomorphism of $R$-algebras
\[
\HGHASym \cong \mathbf{e}_{S_2} ( A \rtimes S_{3} ) \mathbf{e}_{S_2} .
\]


\subsection{Polynomial algebras}
\label{sec: polynomial_rings}
We now specialise to the case where
\[
A = R[x_1 , x_2 , x_3]
\]
is a polynomial algebra over $R$, and $S_3$ acts on $A$ by permuting the variables $x_1,x_2,x_3$.
For $f \in R[x_1,x_2,x_3]$ we will sometimes write $f(x_1,x_2,x_3)$ for $f$, and set
\[
f(x_2,x_1,x_3) := \alpha_{(12)} f \; , \;  f(x_1,x_3,x_2) := \alpha_{(23)} f \; , \ f(x_3,x_2,x_1) := \alpha_{(13)} f 
\]
though we do not think of $f$ as a function.
By the fundamental theorem on elementary symmetric polynomials
\[
A^{S_2} = R[ x_1 + x_2 \, , \,  x_1x_2 \, , \,  x_3 ] , 
\]
which is itself a polynomial algebra in 3 variables.
The $R$-module isomorphism \eqref{eqn: example_S3_R_module_isom} is an isomorphism from $\HGHASym$ to the $R$-module
\[
A^{S_2} \oplus A = R[x_1 + x_2 \, , \, x_1x_2 \, , \, x_3 ] \oplus R[ x_1 , x_2 , x_3 ] .
\]
By Proposition \ref{prop: example_S3_skew_hecke_product} the product on $A^{S_2} \oplus A$ is determined by the following:
\begin{enumerate}
\item The summand $R[ x_1 + x_2 \, , \, x_1x_2 \, , \, x_3 ]$ is a subalgebra.
\item The summand $R[ x_1,x_2,x_3 ]$ is an $R[x_1 + x_2 \, , \, x_1x_2 \, , \, x_3]$-bimodule with actions
\[
f \cdot h \cdot f' = f h (\alpha_{(23)} f') = f(x_1,x_2,x_3) h(x_1,x_2,x_3) f'(x_1,x_3,x_2)
\]
for $f,f' \in A^{S_2}$ and $h \in A$.
\item The product of $h,h' \in R[x_1,x_2,x_3]$ is equal to
\[
\begin{pmatrix}
h(x_1,x_2,x_3) h'(x_1,x_3,x_2) + h(x_2,x_1,x_3) h'(x_3,x_1,x_2) \\
h(x_2,x_1,x_3) h'(x_3,x_2,x_1) .
\end{pmatrix} .
\]
\end{enumerate}
The isomorphism \eqref{eqn: fixed_point_isom_S3_matrix} maps $(f,h) \in A^{S_2} \oplus A$ to the matrix
\[
\begin{pmatrix}
f(x_1,x_2,x_3) & h(x_1,x_3,x_2) & h(x_3,x_1,x_2) \\
h(x_1,x_2,x_3) & f(x_1,x_3,x_2) & h(x_3,x_2,x_1) \\
h(x_2,x_1,x_3) & h(x_2,x_3,x_1) & f(x_3,x_2,x_1) 
\end{pmatrix} .
\]
The action $\alpha$ preserves the natural grading on $R[x_1,x_2,x_3]$ defined by setting the degree of each of $x_1,x_2,x_3$ to be equal to 1.
It follows from \S \ref{sec: grad_filt} that $\HGHA$ is graded, with $\HGHA_{i}$ consisting of elements
\[
f \otimes S_2 + h \otimes (23)S_2 + \alpha_{(12)} h \otimes (13)S_2
\]
such that $f \in R[x_1,x_2,x_3]^{S_2}_{i}$ and $h \in R[x_1,x_2,x_3]_{i}$.

\subsection{Functions on $S_3$}
\label{sec: example_functions_on_S3}
Let $A = RS_3$, the commutative $R$-algebra of $R$-valued functions on $S_3$, equipped with the action of $S_3$ given by $( \alpha_{g} f ) (k) = f (g^{-1}k)$.
By Proposition \ref{prop: stone} there is an $R$-algebra isomorphism
\[
\mathcal{H}_{R}(S_3,S_2,RS_3,\alpha) \cong \End_{R} ( \Ind_{S_2}^{S_3} R ) \cong M_{2}(R)
\]
because $\Ind_{S_2}^{S_3} R \cong R[S_3/S_2]$ is a free $R$-module of rank 2.

\subsection{Semidirect products}
\label{sec: example_semidirect}
Let $K$ be a finite group and $\alpha$ be an action of $S_3$ on $K$ by group automorphisms.
Denote by the same symbol $\alpha$ the corresponding action of $S_3$ on $R[K]$.
It follows from Proposition \ref{prop: skew_hecke_semi_direct} that there is an $R$-algebra isomorphism
\[
\mathcal{H}_{R}(S_3,S_2,R[K],\alpha) \cong \mathcal{H}_{R}( K \rtimes S_{3} , S_2 ) .
\]
By Theorem \ref{thm: fixed_points} (or the statement for non-skew Hecke algebras \cite[Chap.\ I, Thm.\ 4.8]{KriegBook1990}) the $R$-algebra $\mathcal{H}_{R}( K \rtimes S_{3} , S_{2} )$ is isomorphic to the endomorphism algebra of the $R[K \rtimes S_3]$-module $\Ind_{S_2}^{K \rtimes S_3} R \cong \Ind_{S_3}^{K \rtimes S_3} \Ind_{S_2}^{S_3} R$.
\\
\\
For example, suppose that $L$ is a finite group, $K = L^{3}$, and $\alpha$ is the action of $S_3$ on $L^{3}$ given by permuting the factors.
Then $S_3$ acts on $R[L^{3}] \cong R[L]^{\otimes 3}$ by permuting the tensor factors, $K^{3} \rtimes S_3 = L \wr S_3$ is a wreath product group, and
\[
\mathcal{H}_{R}(S_3,S_2,R[L]^{\otimes 3},\alpha) \cong \mathcal{H}_{R}( L \wr S_{3} , S_2 ) \cong \End_{R[G]} ( \Ind_{S_2}^{L \wr S_3} R ) .
\]


\bibliography{skewhecke}{}
\bibliographystyle{amsalpha}


\end{document}